\documentclass[11pt,a4paper]{article}
\usepackage[utf8]{inputenc}
\usepackage[T1]{fontenc}
\usepackage{amsmath}
\usepackage{amsfonts}
\usepackage{amssymb}
\usepackage{amsthm}
\usepackage{bm}
\usepackage{enumitem}
\theoremstyle{plain}
\newtheorem{theorem}{Theorem}
\newtheorem{assumption}[theorem]{Assumption}

\newtheorem{proposition}[theorem]{Proposition}
\newtheorem{definition}[theorem]{Definition}
\newtheorem{lemma}[theorem]{Lemma}
\newtheorem{remark}[theorem]{Remark}
\newtheorem*{remarque*}{Remark}

\newcommand{\dr}{\mathrm{d}}
\newcommand{\sca}{\mathfrak{s}}
\newcommand{\crule}[3][c]{%
    \par\noindent
    \makebox[\linewidth][#1]{\rule{#2\linewidth}{#3}}}

\usepackage[left=2.5cm,right=2.5cm,top=2cm,bottom=2cm]{geometry}
\usepackage{graphicx}
\usepackage{subcaption}
\usepackage{float}
\usepackage{verbatim}
\usepackage[colorlinks=true,
            linkcolor=black,
            urlcolor=black,
            citecolor=black]{hyperref}
\usepackage{cite}
\usepackage{dsfont}
\usepackage[multiple]{footmisc}

\title{\bf{Stable limit theorems for additive functionals of one-dimensional diffusion processes}}
\author{Loïc Béthencourt\thanks{Sorbonne Université, CNRS, Laboratoire de Probabilités, Statistique et Modélisation, F-75005 Paris, France. Email: loic.bethencourt@sorbonne-universite.fr.}\thanks{Acknowledgements: I would like to thank Nicolas Fournier for all the fruitful discussions and advice.}}
\date{}

\begin{document}

\maketitle
\begin{abstract}
    We consider a positive recurrent one-dimensional diffusion process with continuous coefficients and we establish stable central limit theorems for a certain type of additive functionals of this diffusion. In other words we find some explicit conditions on the additive functional so that its fluctuations behave like some $\alpha$-stable process in large time for $\alpha\in(0,2]$.
\end{abstract}

\crule{.75}{0.5pt}
\noindent
2010 \textit{Mathematics Subject Classification:} 60J60, 60F05.

\noindent
\textit{Keywords and phrases: One-dimensional diffusion processes, Stable central limit theorem, Stable processes, Local times.}

\section{Introduction and main result}

Consider a one-dimensional diffusion process  with continuous coefficients $b$ and $\sigma$, i.e. a continuous adapted process $(X_t)_{t\geq0}$ satisfying the SDE
\begin{equation}\label{SDE}
    X_t = \int_0^t b(X_s)\dr s + \int_0^t\sigma(X_s)\dr B_s,
\end{equation}

\noindent
where $(B_t)_{t\geq0}$ is a Brownian motion. Without loss of generality, we assume $X_0 = 0$, even if it means changing $b$ and $\sigma$. If $\sigma$ does not vanish, then weak existence and uniqueness in law hold for \eqref{SDE}, see Kallenberg \cite[Chapter 23 Theorem 23.1]{kallenberg2006foundations}. 

\begin{assumption}\label{assump_b_sigma}
The functions $b,\sigma:\mathbb{R}\to\mathbb{R}$ are continuous and for all $x\in\mathbb{R}$, $\sigma(x) > 0$. Moreover $b$ and $\sigma$ are such that the process $(X_t)_{t\geq0}$ is positive recurrent in the sense of Harris.
\end{assumption}

\noindent
We recall that a strong Markov process $(\Omega, \mathcal{F}, (\mathcal{F}_t)_{t\geq0}, (X_t)_{t\geq0}, (\mathbb{P}_x)_{x\in\mathbb{R}})$ valued in $(\mathbb{R}, \mathcal{B}(\mathbb{R}))$ is said to be recurrent in the sense of Harris if it has a $\sigma$-finite invariant measure $\mu$ such that for all $A\in\mathcal{B}(\mathbb{R})$
$$\mu(A) > 0 \quad \text{implies} \quad \limsup_{t\to\infty}\bm{1}_A(X_t) = 1, \qquad\mathbb{P}_x-\text{almost surely for all }x\in\mathbb{R}.$$
\noindent
The process is then said to be positive recurrent if $\mu(\mathbb{R}) < \infty$ and null recurrent otherwise. We introduce the scale function $\sca$ of $(X_t)_{t\geq0}$ defined by
\begin{equation}\label{scale_function}
    \sca(x) = \int_0^x \exp\left(-2\int_0^v b(u)\sigma^{-2}(u)\dr u\right)\dr v,
\end{equation}
which is a $C^2$, strictly increasing function solving $\frac{1}{2}s''\sigma^2 + s'b = 0$. We also introduce the speed measure density
\begin{equation}\label{speed_measure}
    m(x) =  \sigma^{-2}(x)\exp\left(2\int_0^x b(v)\sigma^{-2}(v)\dr v\right) = [\sigma^{2}(x)\sca'(x)]^{-1}
\end{equation}

\begin{remark}\label{equivalence_harris}
Assume $b$ and $\sigma$ are continuous and $\sigma$ does not vanish. Then we have the equivalence between the two following propositions, see for instance Revuz-Yor \cite[Chapter VII]{revuz2013continuous}.

\begin{enumerate}
    \item $(X_t)_{t\geq0}$ is positive recurrent in the sense of Harris.
    \item $\lim_{x\to\pm\infty}\sca(x) = \pm\infty$ and $\kappa^{-1}:= \int_{\mathbb{R}}m(x)\dr x < \infty$.
\end{enumerate}

\noindent
Moreover, in this case, the measure $\mu(\dr x) = \kappa m(x)\dr x$ is the unique invariant probability measure for the process $(X_t)_{t\geq0}$.

\end{remark}

The ergodic theorem for Harris recurrent processes, see Azema-Duflo-Revuz \cite{azema1969mesure} or Revuz-Yor \cite[Chapter X]{revuz2013continuous}, tells us that for $f\in\mathrm{\bm{L}}^1(\mu)$, a.s.
\begin{equation}\label{ergodic_theorem}
    \frac{1}{t}\int_0^t f(X_s)\dr s \underset{t\to\infty}{\longrightarrow} \mu(f),
\end{equation}

\noindent
where $\mu(f) = \int_{\mathbb{R}} f\dr \mu$. The convergence in \eqref{ergodic_theorem} can be seen as a strong law of large numbers for the additive functional $\int_0^t f(X_s)\dr s$ of the process $(X_t)_{t\geq0}$. Then it is very natural to study its fluctuations, i.e. to describe the asymptotic behavior of $\frac{1}{t}\int_0^t (f(X_s) - \mu(f))\dr s$. In this paper, we give simple conditions on $f$ for these asymptotic fluctuations to be $\alpha$-stable for some $\alpha\in(0,2]$. The conditions on $f$ are entirely prescribed by the coefficients $b$ and $\sigma$.

\medskip
We recall that $\ell:[0,\infty)\mapsto(0,\infty)$ is said to be slowly varying if for every $\lambda>0$, $\ell(\lambda x) / \ell(x) \to 1$ as $x$ goes to infinity. The use of such functions is justified by the fact that domains of attraction of stable laws involve slowly varying functions.

\begin{assumption}\label{assump_f}
The function $f: \mathbb{R}\mapsto\mathbb{R}$ is locally bounded and Borel. Moreover, there exists $\alpha>0$, $(f_+, f_-)\in\mathbb{R}^2$ and a continuous slowly varying function $\ell:[0,\infty)\mapsto(0,\infty)$ such that
$$[\sigma(x)\sca'(x)]^{-2}|\sca(x)|^{2-1/\alpha}\ell(|\sca(x)|)f(x) \underset{x\to\pm\infty}{\longrightarrow}f_{\pm},$$

\noindent
and $|f_+| + |f_-| > 0$. If $f\in \mathrm{\bm{L}}^1(\mu)$, we impose $\mu(f) = 0$. Finally, if $\alpha \geq2$, we impose $f$ to be continuous.
\end{assumption}

This assumption appears naturally in the computations. We refer to Section \ref{section_applications} for a collection of concrete applications and examples. In the critical stable regime $\alpha = 1$, we will add a mild assumption which ensures that the set $\mathrm{\bm{L}}^1(\mu)$ is big enough.

\begin{assumption}\label{assump_mu_integ}
There exists $\lambda > 0$ such that $|\sca|^{\lambda}\in \mathrm{\bm{L}}^1(\mu)$.
\end{assumption}

Under Assumption \ref{assump_f}, we set $\rho = \int_1^{\infty}\left(\int_x^{\infty}\frac{\dr v}{v^{3/2}\ell(v)}\right)^2\dr x$ which can be infinite, and we define the diffusive constant as follows:

\begin{itemize}[leftmargin=*]
    \item If $\alpha > 2$ or $\alpha = 2$ and $\rho < \infty$, $\sigma_{\alpha}^2 = 4\kappa\int_{\mathbb{R}}\sca'(x)\left(\int_x^{\infty}f(v)[\sigma^{2}(v)\sca'(v)]^{-1}\dr v\right)^2 \dr x $.
    
    \item If $\alpha = 2$ and $\rho = \infty$, $\sigma_{2}^2 = 4\kappa(f_+^2 + f_-^2)$.
    
    \item If $\alpha\in(0,2)$, $\sigma_{\alpha}^{\alpha} = \frac{\kappa2^{\alpha-2}\pi}{\alpha\sin(\alpha\pi/2)}\left(\frac{\alpha^{\alpha}}{\Gamma(\alpha)}\right)^2 (|f_+|^{\alpha} + |f_-|^{\alpha})$.
\end{itemize}

\noindent
We also introduce, for $\xi\in\mathbb{R}$, the complex numbers

\begin{itemize}[leftmargin=*]
    \item $z_{\alpha}(\xi) = 1-i\frac{\mathrm{sgn}(f_+)|f_+|^{\alpha} + \mathrm{sgn}(f_-)|f_-|^{\alpha}}{|f_+|^{\alpha} + |f_+|^{\alpha}}\tan\left(\frac{\alpha\pi}{2}\right) \mathrm{sgn}(\xi)$ if $\alpha\in(0,2)\setminus\{1\}$.
    
    \item $z_1(\xi) = 1+i\frac{f_+ + f_-}{|f_+| + |f_-|}\frac{2}{\pi}\mathrm{sgn}(\xi)\left[\log\left(\frac{\pi|\xi|}{2(|f_+| + |f_-|)}\right) + 2\gamma + \log(2) + \frac{\log|f_+|f_+ + \log|f_-|f_- }{f_+ + f_-}\right]$, where $\gamma$ is the Euler constant. 
\end{itemize}

Finally, for a family of processes $((Y_t^{\epsilon})_{t\geq0})_{\epsilon > 0}$ valued in $\mathbb{R}$, we say that $(Y_t^{\epsilon})_{t\geq0} \overset{f.d.}{\longrightarrow} (Y_t^0)_{t\geq0}$ if for all $n\geq1$ and all $0<t_1 < \cdots <t_n$, the vector $(Y_{t_i}^{\epsilon})_{1\leq i\leq n}$ converges in law to $(Y_{t_i}^0)_{1\leq i\leq n}$ in $\mathbb{R}^n$. We are now ready to state our main theorem, which concerns positive recurrent diffusions.

\begin{theorem}\label{main_theorem}
Suppose Assumptions \ref{assump_b_sigma} and \ref{assump_f}. Let $(X_t)_{t\geq0}$ be a solution of \eqref{SDE}, $(W_t)_{t\geq0}$ a Brownian motion and $(S_t^{(\alpha)})_{t\geq0}$ an $\alpha$-stable process such that $\mathbb{E}[\exp(i\xi S_t^{(\alpha)})] = \exp(-t|\xi|^{\alpha}z_{\alpha}(\xi))$.

\begin{enumerate}[label=(\roman*)]
    \item If $\alpha > 2$, or $\alpha = 2$ and $\rho<\infty$,
    $$\left(\epsilon^{1/2}\int_0^{t/\epsilon}f(X_s)\dr s\right)_{t\geq0} \overset{f.d.}{\longrightarrow} \left(\sigma_{\alpha}W_t\right)_{t\geq0} \qquad\hbox{as $\epsilon\to0$}.$$

    \item If $\alpha = 2$ and $\rho=\infty$, we set $\rho_{\epsilon} = \int_1^{1/\epsilon}\left(\int_x^{\infty}\frac{\dr v}{v^{3/2}\ell(v)}\right)^2\dr x$ and we have
    $$\left(|\epsilon/\rho_{\epsilon}|^{1/2}\int_0^{t/\epsilon}f(X_s)\dr s\right)_{t\geq0} \overset{f.d.}{\longrightarrow}\left(\sigma_{2}W_t\right)_{t\geq0}\qquad\hbox{as $\epsilon\to0$}.$$

    \item If $\alpha\in(0,2)\setminus\{1\}$,
    $$\left(\epsilon^{1/\alpha}\ell(1/\epsilon)\int_0^{t/\epsilon}f(X_s)\dr s\right)_{t\geq0} \overset{f.d.}{\longrightarrow} \left(\sigma_{\alpha}S_t^{(\alpha)}\right)_{t\geq0} \qquad\hbox{as $\epsilon\to0$}.$$
    
    \item If $\alpha=1$ and Assumption \ref{assump_mu_integ} holds, there exists a deterministic family $(\xi_{\epsilon})_{\epsilon >0}$ of real numbers such that
    $$\left(\epsilon \ell(1/\epsilon) \int_0^{t/\epsilon}f(X_s)\dr s - \xi_{\epsilon} t\right)_{t\geq0} \overset{f.d.}{\longrightarrow} \left(\sigma_{1}S_t^{(1)}\right)_{t\geq0} \qquad\hbox{as $\epsilon\to0$}.$$
    
    \noindent
     Moreover $\xi_{\epsilon} \underset{\epsilon\to0}{\sim}\kappa(f_+ + f_-)\ell(1/\epsilon)\zeta_{\epsilon}$, where  $\zeta_{\epsilon} =-\int_{1/\epsilon}^{\infty}\frac{\dr x}{x\ell(x)}$ if $f\in \mathrm{\bm{L}}^1(\mu)$ and  $\zeta_{\epsilon} =\int_{1}^{1/\epsilon}\frac{\dr x}{x\ell(x)}$ otherwise.
\end{enumerate}
\end{theorem}

For the reader more familiar with infinitesimal generators than with characteristic functions, let us mention the following remark.

\begin{remark}
When $\alpha\in (0,2)$, the process $(\sigma_{\alpha}S_t^{(\alpha)})_{t\geq0}$ of Theorem \ref{main_theorem} is a L\'evy process with L\'evy measure
$$\nu(\dr x) = \left(c_+x^{-1-\alpha}\bm{1}_{\{x > 0\}} + c_-|x|^{-1-\alpha}\bm{1}_{\{x < 0\}}\right)\dr x,$$

\noindent
where
$$c_+ =\lambda_{\alpha}\left[\bm{1}_{\{f_+>0\}}|f_+|^{\alpha} + \bm{1}_{\{f_->0\}}|f_-|^{\alpha}\right],\quad c_- =\lambda_{\alpha}\left[\bm{1}_{\{f_+<0\}}|f_+|^{\alpha} + \bm{1}_{\{f_-<0\}}|f_-|^{\alpha}\right]$$

\noindent
and $\lambda_{\alpha} = \frac{\kappa2^{\alpha-2}\pi}{\alpha\sin(\alpha\pi/2)}\left(\frac{\alpha^{\alpha}}{\Gamma(\alpha)}\right)^2$. Its infinitesimal generator $\mathcal{L}^{(\alpha)}$ is such that for all $\phi\in C^{\infty}_c(\mathbb{R})$,

\begin{itemize}[leftmargin=*]
    \item $\mathcal{L}^{(\alpha)}\phi(x) = \int_{\mathbb{R}\setminus\{0\}}[\phi(x+z) - \phi(x)]\nu(dz)$ if $\alpha\in(0,1)$,
    
    \item $\mathcal{L}^{(\alpha)}\phi(x) = \int_{\mathbb{R}\setminus\{0\}}[\phi(x+z) - \phi(x) - z\phi'(x)]\nu(dz)$ if $\alpha\in(1,2)$,
    
    \item if $\alpha=1$, $\mathcal{L}^{(\alpha)}\phi(x) = a\phi'(x) + \int_{\mathbb{R}\setminus\{0\}}[\phi(x+z) - \phi(x) - z\bm{1}_{\{|z|\leq1\}}\phi'(x)]\nu(dz)$, where
    $$a = -(f_+ + f_-)\left[ 2\gamma + \log(2) + \kappa\log(\kappa) + \frac{\log|f_+|f_+ + \log|f_-|f_- }{f_+ + f_-} + \kappa\frac{\pi}{2}A\right],$$
    \noindent
    $\gamma$ is the Euler constant and $A = \int_{0}^{1}\frac{\sin(x) - x}{x^2}dx + \int_{1}^{\infty}\frac{\sin(x)}{x^2}dx$.
\end{itemize}
\noindent
One can refer to Bertoin \cite[page 24]{bertoin1996levy} to understand how to go from the L\'evy-Khintchine triplet to the generator and to Sato \cite[Chapter 3.14]{ken1999levy} to go from the L\'evy-Khintchine triplet to the characteristic function.
\end{remark}

Observe that in the L\'evy regime $\alpha \in (0,2)$, we have a convergence in finite dimensional distribution of a continuous process towards a discontinuous process (stable processes possess many jumps). Hence we cannot hope to obtain convergence in law as a process, for example for the usual Skorokhod distance. Observe also that if $f_- = 0$, the stable process has only positive jumps. Thus the limiting process is completely asymmetric although the process $(f(X_t))_{t\geq0}$ may visit the whole space infinitely often.

\medskip
When $f\in\mathrm{\bm{L}}^1(\mu)$ and $\mu(f)\neq0$, we can use Theorem \ref{main_theorem} with the function $f-\mu(f)$ provided it satisfies Assumption \ref{assump_f}.

\medskip
When $\alpha\geq2$, the stable process obtained in the limit is a Brownian motion, which is the only $2$-stable process. A standard strategy to show central limit theorems for the additive functional $\int_0^t f(X_s)\dr s$ is to solve the Poisson equation  $\mathcal{L}g = f$, where $\mathcal{L}$ is the infinitesimal generator of $(X_t)_{t\geq0}$ with domain $D_{\mathcal{L}}$. One can refer to Jacod-Shiryaev \cite[Chapter VIII]{jacod2013limit}, Pardoux-Veretennikov \cite{pardoux} or Cattiaux, Chafaï and Guillin \cite{cattiaux2011central}. Assume that $f$ is a continuous function in $\mathrm{L}^1(\mu)$ and define the function
\begin{equation}\label{solution_Poisson}
    g(x) = 2\int_0^x \sca'(v)\int_v^{\infty}f(u)[\sigma^{2}(u)\sca'(u)]^{-1}\dr u\dr v.
\end{equation}

\noindent
Then $g$ is a $C^2$ function solving the Poisson equation $2bg' + \sigma^2 g'' = -2f$. Hence, applying the Itô formula with $g$, we express the additive functional as a martingale plus some remainder, and hope to obtain the result using a central limit theorem for martingales. This will be the strategy used when $\alpha \geq 2$, and we will actually prove the following more general result. Note that, for $g$ to be well defined, it suffices that $f\in\mathrm{L}^1(\mu)$ since $\mu(\dr x) = \kappa[\sigma^{2}(u)\sca'(u)]^{-1} \dr x$, whence $\int_v^{\infty}f(u)[\sigma^{2}(u)\sca'(u)]^{-1}\dr u = \kappa^{-1}\mu(f\bm{1}_{(v,\infty)})$.

\begin{theorem}\label{general_theorem}
Suppose Assumption \ref{assump_b_sigma} and let $(X_t)_{t\geq0}$ be the solution of \eqref{SDE}. Let also $f$ be a continuous function in $\mathrm{\bm{L}}^1(\mu)$ such that $\mu(f) = 0$ and $g$ be defined by \eqref{solution_Poisson}. If $g'\sigma\in \mathrm{\bm{L}}^2(\mu)$, then
$$\left(\epsilon^{1/2}\int_0^{t/\epsilon}f(X_s)\dr s \right)_{t\geq0} \overset{f.d.}{\longrightarrow} \left(\gamma W_t\right)_{t\geq0} \qquad\hbox{as $\epsilon\to0$},$$

\noindent
where $(W_t)_{t\geq0}$ is a Brownian motion and $\gamma^2 = \int_{\mathbb{R}}[g'\sigma]^2 \dr \mu$. Moreover, if $g$ is bounded, the finite dimensional convergence can be replaced by a convergence in law of continuous processes, with the topology of uniform convergence on compact time intervals.
\end{theorem}

Note that in the L\'evy regime, we do not require $f$ to be regular whereas we need $f$ to be continuous in the diffusive regime, so we can apply the Itô formula with $g$.

\subsection*{References}

The central limit theorem for Markov processes has a long history which goes back to Langevin \cite{langevin1908theorie}. Roughly, he studied a one-dimensional particle, the velocity of which is subject to random shocks and a restoring force $F(v) = -v$, and showed that the position of the particle behaves like a Brownian motion, when $t$ tends to infinity. 

\medskip
Probabilistic techniques to obtain such results can be found in Jacod-Shiryaev \cite[Chapter VIII]{jacod2013limit}: if there exists a solution to the Poisson equation having few properties, then we have a central limit theorem. This strategy is applied in Pardoux-Veretennikov \cite{pardoux} for very general multidimensional diffusion processes. However, \cite{pardoux} does not completely include the diffusive regime of Theorem \ref{main_theorem}: we have a little less assumptions but our result only holds in dimension one. It seems that the critical diffusive regime $\alpha = 2$ was not much studied. Finally, one can refer to Cattiaux, Chafaï and Guillin \cite{cattiaux2011central} for a detailed state of the art of the techniques used to show such theorems.

\medskip
Regarding the L\'evy regime, we could not find many results for Markov processes and this does not seem to be well developed. Through the prism of ergodic theory, Gouëzel \cite{gouezel2004central} proved stable limit theorems for a certain class of maps. Applied to a positive recurrent Markov chain $(X_n)_{n\geq0}$ with countable state space, its results tells us that if the function $f$ is such that, up to a slowly varying function,
\begin{equation}\label{gouezel_condition}
    \mathbb{P}_i\left(\sum_{k=0}^{\tau_i-1}f(X_k) > x\right) \underset{x\to\infty}{\sim}c_+|x|^{-\alpha} \quad \text{and}\quad\mathbb{P}_i\left(\sum_{k=0}^{\tau_i-1}f(X_k) < x\right)\underset{x\to-\infty}{\sim}c_-|x|^{-\alpha},
\end{equation}

\noindent
where $\tau_i = \inf\{k>0, X_k = i\}$, then we have
$$\frac{1}{n^{\alpha}}\sum_{k=0}^{n-1}f(X_k) \overset{d}{\longrightarrow}S^{\alpha},$$

\noindent
where $S^{\alpha}$ is a stable random variable. This is more or less obvious in this case since we have a sum of i.i.d blocks in the domain of attraction of stable laws and we can thus apply classical stable central limit theorems.

\medskip
Similar results with more tangible assumptions were proved by Jara, Komorowski and Olla \cite{ollamilton} for Markov chains with general state space. They assume the function $f$ is such that, roughly, $\mu(\{|f|>x\}) \simeq |x|^{-\alpha}$, where $\mu$ is the invariant measure of the chain and a spectral gap condition, the latter one being more or less equivalent to return times having exponential moments. Since the return times are small, it is reasonable to think that we have $\mathbb{P}_i(\sum_{k=0}^{\tau_i-1}f(X_k) > x)\simeq\mu(\{f>x\})$, at least if we consider a countable state space. They actually give two proofs, one using a martingale approximation and a second one with a renewal method involving a coupling argument.

\medskip
Mellet, Mischler and Mouhot \cite{mellet2011fractional} showed that for a linearized Boltzmann equation with heavy-tailed invariant distribution, the rescaled distribution of the position converges to the solution of the fractional heat equation, i.e. the position process behaves like a stable process. Their work is closely related to \cite{ollamilton} although their proof is entirely analytic.

\medskip
The method proposed in this article is rather powerful and gives a simple condition on the additive functional for stable limit theorems to occur. The assumption made on the function $f$ does not seem to be a time-continuous equivalent of \eqref{gouezel_condition} and we do not assume anything on the return times, nor on $\mu(\{|f|>x\})$. We will see in Section \ref{section_applications} that, indeed, the index $\alpha$ is not always prescribed by $\mu(\{|f|>x\})$.

\medskip
The strategy is the following: we classically express the process $(X_t)_{t\geq0}$ as a Brownian motion $(W_t)_{t\geq0}$ changed in time and in space. Then we write $\int_0^t f(X_s)\dr s = \int_0^{\rho_t}\phi(W_s)\dr s$ for a certain function $\phi$, which roughly looks like $\int_0^{\tau_t}|W_s|^{1/\alpha - 2}\dr s$ (up to asymmetry and principal values issues) in large time, where $(\tau_t)_{t\geq0}$ is the generalized inverse of the local time at $0$ of $(W_t)_{t\geq0}$. This process is known to be a stable process, see Itô-McKean \cite[page 226]{ito2012diffusion}, Jeulin and Yor \cite{jeulin1981distributions} and Biane and Yor \cite{biane1987valeurs}. The computations are thus tractable and the method is very robust. 

\medskip
This method was proposed by Fournier and Tardif \cite{fournier2018one}, studying a kinetic model and the purpose of this paper is to extend this method to general one-dimensional diffusions and to more general additive functionals. The model they were studying was the following:
\begin{equation}\label{kinetic_model}
V_t = B_t -\frac{\beta}{2}\int_0^t \frac{V_s}{1 + V_s^2}\dr s, \qquad X_t = \int_0^t V_s \dr s,
\end{equation}

\noindent
where $(B_t)_{t\geq0}$ is a Brownian motion and $\beta > 1$. The processes $(V_t)_{t\geq0}$ and $(X_t)_{t\geq0}$ are respectively the velocity and the position of a one-dimensional particle subject to the restoring force $F(v) = - \frac{\beta}{2}\frac{v}{1+v^2}$ and random shocks. Lebeau and Puel \cite{lebeau2019diffusion} showed that, when $\beta\in(1,5)\setminus\{2, 3, 4\}$, the rescaled distribution of the position $X_t$ converges to the solution of the fractional heat equation. Their work is analytical and relies on a deep spectral analysis leading to an impressive result. The diffusive regime ($\beta > 5$) and critical diffusive regime ($\beta = 5$) are treated by Nasreddine-Puel \cite{nasreddine2015diffusion} and Cattiaux-Nasreddine-Puel \cite{cattiaux2019diffusion}. We stress that all these papers are P.D.E papers. 

\medskip
To summarize, if the restoring force field is strong enough ($\beta \geq5$), the rescaled position process resembles a Brownian motion and we say that there is a \textit{normal diffusion limit}. On the other hand, when the force is weak ($\beta \in (1,5)$), the position resembles a stable process and there is an \textit{anomalous diffusion limit} (or \textit{fractional diffusion limit}).

\medskip
Actually, physicists discovered that atoms diffuse anomalously when they are cooled by a laser. See for instance Castin, Dalibard and Cohen-Tannoudji \cite{castin1991limits},  Sagi, Brook, Almog and Davidson \cite{sagi2012observation} and Marksteiner, Ellinger and Zoller \cite{marksteiner1996anomalous}. A theoretical study (see Barkai, Aghion and Kessler \cite{barkai2014area}) modeling the motion of atoms precisely by \eqref{kinetic_model} proved with quite a high level of rigor the observed phenomenons.

\medskip
In the L\'evy regime, the index $\alpha$ is not prescribed by the function $f = \text{id}$ and the invariant measure $\mu$. Indeed observe that we have $\alpha = (\beta + 1) / 3$ whereas one can check, see Section \ref{section_applications}, that
$$\mu(\{|f|>x\})\simeq|x|^{1-\beta}.$$

\noindent
In this situation, the return times are large and the dynamics are more complex.

\medskip
Then, using probabilistic techniques, Fournier and Tardif \cite{fournier2018one} treated all cases of \eqref{kinetic_model} for a larger class of symmetric forces. Naturally, we recover their result and enlarge it to asymmetrical forces, see Section \ref{section_applications}. In a companion paper \cite{fournier_dimension}, they also prove the result in any dimension and the proof is much more involved.

\medskip
We also found a paper of Bertoin and Werner \cite{windings_ber_wer} where they use similar techniques to redemonstrate \textit{Spitzer's theorem}.

\subsection*{Plan of the paper}

In Section \ref{preliminaries}, we recall some facts about slowly varying functions, local times and generalize some results on stable processes found in \cite{ito2012diffusion}, \cite{jeulin1981distributions} and \cite{biane1987valeurs}. Section \ref{proof} is dedicated to the proof of Theorem \ref{main_theorem} and Theorem \ref{general_theorem}. Finally, we apply our result to several models in Section \ref{section_applications}, illustrating some remarks made in the previous section.

\subsection*{Notations} 
Throughout the paper, we will use the function $\mathrm{sgn}_{a,b}(x) = a\bm{1}_{\{x>0\}} + b\bm{1}_{\{x<0\}}$ where $(a,b)\in\mathbb{R}^2$. The function $\mathrm{sgn}$ will denote the usual sign function, i.e. $\mathrm{sgn}(x) = \mathrm{sgn}_{1,-1}(x)$.

\section{Preliminaries}\label{preliminaries}

\subsection{Slowly varying functions}\label{slowly_varying_section}

In this section, we define slowly varying functions and give the properties of these functions that will be useful to us. One can refer to Bingham, Goldie and Teugels \cite{bingham1989regular} for more details.

\begin{definition}
A measurable function $\ell:[0,\infty)\mapsto(0,\infty)$ is said to be slowly varying if for every $\lambda>0$, $\ell(\lambda x) / \ell(x) \longrightarrow 1$ as $x\to\infty$.

\end{definition}

\noindent
Famous examples of slowly varying functions are powers of logarithm, iterated logarithms or functions having a strictly positive limit at infinity. It holds that for any $\theta > 0$, $x^{\theta}\ell(x)\to\infty$ and $x^{-\theta}\ell(x)\to0$, as $x\to\infty$. We will need the following lemma.

\begin{lemma}\label{prop_slow_var}
Let $\ell$ be a continuous slowly varying function.

\begin{enumerate}[label=(\roman*)]
    \item For each $0<a < b$, we have $\sup_{\lambda\in[a,b]}\left|\ell(\lambda x)/\ell(x) - 1\right|\underset{x\to\infty}{\longrightarrow}1.$
    
    \item \text{(Potter's bound).} For any $\delta > 0$ and $C>1$, there exists $x_0 > 0$, such that for all $x,y\geq x_0$,

    $$\frac{\ell(y)}{\ell(x)}\leq C\left(\Big|\frac{y}{x}\Big|^{\delta} \vee \Big|\frac{x}{y}\Big|^{\delta}\right).$$
    
    \item The functions $L : x\mapsto \int_1^x \frac{\dr v}{v\ell(v)}$ and $M : x\mapsto \int_1^x (\int_v^{\infty}\frac{\dr u}{u^{3/2}\ell(u)})^2\dr v$ are slowly varying, as well as $N:x\mapsto\int_x^{\infty} \frac{\dr v}{v\ell(v)}$ if $\int_1^{\infty}\frac{\dr v}{v\ell(v)} < \infty$.
\end{enumerate}
\end{lemma}

\begin{proof} \textbf{Items (i) and (ii).} See \cite[pages 6 and 25]{bingham1989regular}.

\bigskip\noindent
\textbf{Item (iii).} Let us start with the function $L$. For $\lambda > 0$, we have
$$L(\lambda x) = \int_1^{\lambda x} \frac{\dr v}{v\ell(v)} = L(x) + \int_x^{\lambda x}\frac{\dr v}{v\ell(v)} = L(x) + \int_1^{\lambda}\frac{\dr u}{u\ell(xu)}$$

\noindent
where we used the substitution $v = xu$ in the last equality. By (i), $\sup_{u\in[1,\lambda]}|\ell(x) / \ell(ux) - 1| \longrightarrow 0$ as $x\to\infty$, so that $\int_1^{\lambda}\frac{\dr u}{u\ell(xu)} \sim \log(\lambda)/\ell(x)$ as $x$ tends to infinity and it remains to show that $\ell(x)L(x) \longrightarrow \infty$ at infinity. Let $A > 1$ and consider $x > A$: we have

$$\ell(x)L(x) = \ell(x)\int_{1}^x\frac{\dr v}{v\ell(v)} = \ell(x)\int_{1/x}^1\frac{\dr u}{u\ell(xu)} \geq \int_{1/A}^1\frac{\ell(x)\dr u}{u\ell(xu)} \underset{x\to\infty}{\longrightarrow}\log(A),$$

\noindent
by Point (i) again. Now let $A$ to $\infty$ and the result follows. Regarding the function $N$, we assume $\int_1^{\infty}\frac{\dr v}{v\ell(v)} < \infty$ and we write, for $\lambda > 0$,
$$N(\lambda x) = N(x) - \int_x^{\lambda x}\frac{\dr v}{v\ell(v)} = N(x) - \int_1^{\lambda}\frac{\dr u}{u\ell(xu)}$$

\noindent
and, as previously, we only need to show that $\ell(x)N(x) \longrightarrow\infty$. We write

$$\ell(x)N(x) = \ell(x)\int_{x}^{\infty}\frac{\dr v}{v\ell(v)} = \int_1^{\infty}\frac{\ell(x)\dr u}{u\ell(xu)} \geq \int_1^A \frac{\ell(x)\dr u}{u\ell(xu)}$$

\noindent
for any $A>1$. But once again, $\int_1^A \frac{\ell(x)\dr v}{v\ell(xv)} \longrightarrow \log(A)$ and the result holds letting $A$ to infinity. Finally we show the result for $M$. We have, for $\lambda  > 0$,
$$M(\lambda x) = M(x) + \int_x^{\lambda x}\left(\int_v^{\infty}\frac{\dr u}{u^{3/2}\ell(u)}\right)^2 \dr v = M(x) + \int_1^{\lambda}\left(\int_y^{\infty}\frac{\dr z}{z^{3/2}\ell(xz)}\right)^2 \dr y$$

\noindent
A little study shows that $\int_1^{\lambda}\left(\int_y^{\infty}\frac{\dr z}{z^{3/2}\ell(xz)}\right)^2 \dr y \sim \int_1^{\lambda}\left(\int_y^{\infty}\frac{\dr z}{z^{3/2}\ell(x)}\right)^2 \dr y = \frac{4\log(\lambda)}{\ell^2(x)}$. It remains to show that $\ell^2(x) M(x) \longrightarrow \infty$. Let $A > 1$ and $x>A$, we have

$$\ell^2(x) M(x) = \int_{1/x}^1\left(\int_y^{\infty}\frac{\ell(x)\dr z}{z^{3/2}\ell(xz)}\right)^2 \dr y  \geq \int_{1/A}^1\left(\int_y^{\infty}\frac{\ell(x)\dr z}{z^{3/2}\ell(xz)}\right)^2 \dr y \underset{x\to\infty}{\longrightarrow}4\log(A)$$

\noindent
and the result follows by letting $A$ to $\infty$.
\end{proof}

\subsection{Brownian local times}\label{section_local_times}

Local times have been widely studied in Revuz-Yor \cite{revuz2013continuous} to which we refer for much more details. Let $(W_t)_{t\geq0}$ be a Brownian motion. We introduce the local time of $(W_t)_{t\geq0}$ at $x\in\mathbb{R}$, which is the process $(L_t^x)_{t\geq0}$ defined by

$$L_t^x = |W_t - x| - |x| - \int_0^t \mathrm{sgn}(W_s - x)\dr W_s.$$

\noindent
The process $(L_t^x)_{t\geq0}$ is continuous and non-decreasing and the random non-negative measure $\dr L_t^x$ on $[0, \infty)$ is a.s. carried by the set $\{t\geq0, W_t = x\}$, which is a.s. Lebesgue-null. We will heavily use the occupation times formula, see \cite[Chapter 6 Corollary 1.6 page 224]{revuz2013continuous}, which tells that a.s., for every $t\geq0$ and for all Borel function $\varphi : \mathbb{R}\to\mathbb{R}_+$,

$$\int_0^t\varphi(W_s)\dr s = \int_{\mathbb{R}}\varphi(x)L_t^x \dr x.$$

\noindent
We will also use the fact that the map $x\mapsto L_t^x$ is a.s. Hölder of order $\theta$, for $\theta\in (0, 1/2)$, uniformly on compact time intervals, see \cite[Corollary 1.8 page 226]{revuz2013continuous}, i.e. that 
\begin{equation}\label{holder_local_times}
    \sup_{(t,x)\in[0, T]\times\mathbb{R}}|x|^{-\theta}\left|L_t^x -L_t^0\right| < \infty \quad \text{a.s.}
\end{equation}

We now introduce $\tau_t = \inf\{u\geq0, L_u^0 > t\}$, the right-continuous generalized inverse of $(L_t^0)_{t\geq0}$. Some properties of $(\tau_t)_{t\geq0}$ will be needed. For all $t>0$,  $\mathbb{P}(\tau_{t_-} < \tau_t) = 0$, which means that $(\tau_t)_{t\geq0}$ has no deterministic jump time and a.s., for all $t\geq0$, $L_{\tau_t}^0 = t$ since $(L_t^0)_{t\geq0}$ is continuous. We will also use that a.s., for all $t \geq0, W_{\tau_t} = 0$ and that a.s., $\tau_0 = 0$.

\medskip
Finally, the processes $(\tau_t, W_t)_{t\geq0}$ and $\left(L_{\tau_t}^x\right)_{x\in\mathbb{R}, t\geq 0}$ enjoy the following scaling property: for all $c > 0$

\begin{equation}\label{scaling}
    (\tau_t, W_t)_{t\geq0} \overset{d}{=} (c^{-2}\tau_{ct}, c^{-1}W_{c^2t})_{t\geq0}\quad\text{and}\quad\left(L_{\tau_t}^x\right)_{x\in\mathbb{R}, t\geq 0}\overset{d}{=}\left(c^{-1}L_{\tau_{ct}}^{xc}\right)_{x\in\mathbb{R}, t\geq 0}
\end{equation}

\subsection{Stable processes}\label{stable_process}

There is a huge litterature on stable processes and we will mainly refer to Sato \cite{ken1999levy}. In this section we generalize some results on stable processes found in Itô-McKean \cite[page 226]{ito2012diffusion}, Jeulin-Yor \cite{jeulin1981distributions} and Biane-Yor \cite{biane1987valeurs}, where they worked in a symmetric or completely asymmetric framework. These kind of results were initiated by L\'evy himself \cite{levy}, where he first proved that $(\tau_t)_{t\geq0}$ is a $1/2$-stable subordinator. Let us first recall a classic result on stable processes, see Sato \cite[Theorem 14.15 page 86 and Definition 14.16 page 87]{ken1999levy}.

\begin{theorem}\label{stable_def_thm}
Let $\alpha\in(0,2)\setminus\{1\}$ and let $(X_t)_{t\geq0}$ be an $\alpha$-strictly stable process, i.e. a L\'evy process enjoying the scaling property $(X_t)_{t\geq0} \overset{d}{=} (c^{-1/\alpha}X_{ct})_{t\geq0}$ for all $c > 0$. Then there exist $\nu \geq 0$ and $\beta\in[-1,1]$ such that for all $\xi\in\mathbb{R}$

$$\mathbb{E}\left[\exp(i\xi X_t)\right] = \exp\left(-\nu t|\xi|^{\alpha}\left(1-i\beta\tan\left(\frac{\alpha\pi}{2}\right) \mathrm{sgn}(\xi)\right)\right).$$

\noindent
If $(X_t)_{t\geq0}$ is a $1$-stable process, i.e. a L\'evy process such that for all $c > 0$, there exists $k_c\in\mathbb{R}$ with $(X_t)_{t\geq0} \overset{d}{=} (c^{-1}X_{ct} + k_c t)_{t\geq0}$, then there exist $\nu \geq 0$, $\beta\in[-1,1]$ and $\tau\in\mathbb{R}$ such that for all $\xi\in\mathbb{R}$

$$\mathbb{E}\left[\exp(i\xi X_t)\right] = \exp\left(-\nu t|\xi|\left(1 + i\beta\frac{2}{\pi} \mathrm{sgn}(\xi)\log|\xi|\right) + it\tau\xi\right).$$

\noindent
In any case, if $(X_t)_{t\geq0}$ has only positive jumps, i.e. if its L\'evy measure is carried by $(0,\infty)$, then $\beta = 1$.
\end{theorem}

\medskip
We now consider a Brownian motion $(W_t)_{t\geq0}$, its local time $(L_t^0)_{t\geq0}$ at 0, and the generalized inverse $(\tau_t)_{t\geq0}$ of its local time. Our goal is to build all possible stable processes in terms of this Brownian motion and its inverse local time. This is important in order to adapt the method of Fournier-Tardif \cite{fournier2018one} who only treated the symmetric case. Let us recall the notation $\mathrm{sgn}_{a,b}(x) = a\bm{1}_{\{x>0\}} + b\bm{1}_{\{x<0\}}$. We have the following result.

\begin{lemma}\label{stable0_1}
Let $\alpha\in(0,1)$, $(a,b)\in\mathbb{R}^2$ and define the process $K_t = \int_0^{t}\mathrm{sgn}_{a,b}(W_s)|W_s|^{1/\alpha - 2}\dr s$. Then $(S_t)_{t\geq0} = (K_{\tau_t})_{t\geq0}$ is strictly $\alpha$-stable and for all $\xi\in\mathbb{R}$ and all $t\geq 0$, we have

$$\mathbb{E}\left[\exp\left(i\xi S_t\right)\right] = \exp\left(-c_{\alpha,a,b}t|\xi|^{\alpha}\left(1-i\beta_{\alpha, a, b}\tan\left(\frac{\alpha\pi}{2}\right) \mathrm{sgn}(\xi)\right)\right),$$

\noindent
where $c_{\alpha,a,b} = \frac{2^{\alpha - 2}\pi}{\alpha\sin(\alpha\pi/2)}\left(\frac{\alpha^{\alpha}}{\Gamma(\alpha)}\right)^2 (|a|^{\alpha} + |b|^{\alpha})$ and $\beta_{\alpha, a, b} = \frac{\mathrm{sgn}(a)|a|^{\alpha} + \mathrm{sgn}(b)|b|^{\alpha}}{|a|^{\alpha} + |b|^{\alpha}}$.
\end{lemma}

When $\alpha\in(0,1)$, stable processes have finite variations and no drift part meaning they are pure jump processes and should be seen this way: $S_t = \sum_{s\in[0,t]}\int_{\tau_{s-}}^{\tau_s}\mathrm{sgn}_{a,b}(W_u)|W_u|^{1/\alpha - 2}\dr u$. The jumping times are the ones of $(\tau_t)_{t\geq0}$ and the size of a jump is equal to the integral of the function $\mathrm{sgn}_{a,b}(x)|x|^{1/\alpha - 2}$ over the corresponding excursion of the Brownian motion.

\medskip
While $(K_t)_{t\geq0}$ is well-defined when $\alpha\in(0,1)$ since $\int_0^T |W_s|^{1/\alpha - 2}\dr s < \infty$ a.s., this is not the case when $\alpha\in[1,2)$ and we have to work a little more.

\begin{lemma}\label{stable1_2}
Let $\alpha\in (1,2)$, $(a,b)\in\mathbb{R}^2$ and define $K_t = \int_{\mathbb{R}}\mathrm{sgn}_{a,b}(x)|x|^{1/\alpha -2}(L_t^x -L_t^0)\dr x$. Then $(S_t)_{t\geq 0} = (K_{\tau_t})_{t\geq 0}$ is strictly $\alpha$-stable and for all $\xi\in\mathbb{R}$ and all $t\geq 0$, we have

$$\mathbb{E}\left[\exp\left(i\xi S_t\right)\right] = \exp\left(-c_{\alpha,a,b}t|\xi|^{\alpha}\left(1-i\beta_{\alpha, a, b}\tan\left(\frac{\alpha\pi}{2}\right) \mathrm{sgn}(\xi)\right)\right),$$

\noindent
where $c_{\alpha,a,b} = \frac{2^{\alpha - 2}\pi}{\alpha\sin(\alpha\pi/2)}\left(\frac{\alpha^{\alpha}}{\Gamma(\alpha)}\right)^2 (|a|^{\alpha} + |b|^{\alpha})$ and $\beta_{\alpha, a, b} = \frac{\mathrm{sgn}(a)|a|^{\alpha} + \mathrm{sgn}(b)|b|^{\alpha}}{|a|^{\alpha} + |b|^{\alpha}}$.
\end{lemma}

The case $\alpha = 1$ is a little more tedious.

\begin{lemma}\label{stable_1}
Let $(a,b)\in\mathbb{R}$ and define the process $K_t = \int_{\mathbb{R}}\mathrm{sgn}_{a,b}(x)|x|^{-1}(L_t^x -L_t^0\bm{1}_{\{|x| \leq 1\}})\dr x$. Then $(S_t)_{t\geq 0} = (K_{\tau_t})_{t\geq 0}$ is $1$-stable. For all $\xi\in\mathbb{R}$ and all $t\geq 0$, we have

$$\mathbb{E}\left[\exp\left(i\xi S_t\right)\right] = \exp\left(-c_{a,b}t|\xi|\left(1+i\beta_{a, b}\frac{2}{\pi}\log(|\xi|) \mathrm{sgn}(\xi)\right) + it\tau_{a,b}\xi\right),$$

\noindent
where $c_{a,b} = \frac{\pi}{2} (|a| + |b|)$, $\beta_{a, b} = \frac{a + b}{|a| + |b|}$, $\tau_{a,b} = -(a + b)\left[ 2\gamma + \log(2) + \frac{\log|a|a + \log|b|b }{a+b}\right]$ and $\gamma$ is the Euler constant
\end{lemma}

To handle the proof, we need first to make the following general remark.

\begin{remark}\label{levy}
Let $\phi : \mathbb{R}\to\mathbb{R}$ be a measurable function such that for all $T>0$, $\int_0^T |\phi(W_s)|\dr s < \infty$ a.s. The process $Z_t = \int_0^{\tau_t}\phi(W_s)\dr s$ is a L\'evy process with respect to the filtration $(\mathcal{F}_{\tau_t})_{t\geq0}$. Indeed, for all $0<s<t$, $Z_t - Z_s = \int_{\tau_s}^{\tau_t}\phi(W_u)\dr u = \int_{0}^{\tau_t - \tau_s}\phi(W_{u+\tau_s})\dr u = \int_{0}^{\bar{\tau}_{t-s}}\phi(\bar{W}_u)\dr u$, where $(\bar{W}_u)_{u\geq0} = (W_{u+\tau_s})_{u\geq0}$ is a Brownian motion independent of $\mathcal{F}_{\tau_s}$ thanks to the strong Markov property and since $\bar{W}_0 = W_{\tau_s} = 0$ a.s., and where $(\bar{\tau}_u)_{u\geq0}$ is the generalized inverse of its local time. Hence the increments of $(Z_t)_{t\geq0}$ are independent and stationary.
\end{remark}

\begin{proof}[Proof of Lemma \ref{stable0_1}]

Let us first show that $\left(S_t\right)_{t\geq0}$ is strictly $\alpha$-stable. It is a L\'evy process by Remark \ref{levy}. For all $c>0$, we have from \eqref{scaling}

\begin{equation*}
    \begin{split}
        S_t & \overset{d}{=} \int_0^{c^{-2}\tau_{ct}}\mathrm{sgn}_{a,b}(c^{-1}W_{c^2s})|c^{-1}W_{c^2s}|^{1/\alpha - 2}\dr s \\
          & = c^{2-1/\alpha}\int_0^{c^{-2}\tau_{ct}}\mathrm{sgn}_{a,b}(W_{c^2s})|W_{c^2s}|^{1/\alpha - 2}\dr s \\
          & = c^{-1/\alpha}\int_0^{\tau_{ct}}\mathrm{sgn}_{a,b}(W_u)|W_{u}|^{1/\alpha - 2}\dr u,
    \end{split}
\end{equation*}

\noindent
which equals $c^{-1/\alpha}S_{ct}$. Hence $\left(S_t\right)_{t\geq0}$ is strictly $\alpha$-stable. We now focus on showing the explicit form of its characteristic function and introduce first the processes

$$C_t = \int_0^{\tau_t}\bm{1}_{\{W_s > 0\}}|W_s|^{1/\alpha - 2}\dr s, \qquad D_t = \int_0^{\tau_t}\bm{1}_{\{W_s < 0\}}|W_s|^{1/\alpha - 2}\dr s$$

\noindent
and $H_t = C_t - D_t$. We have $S_t = aC_t + bD_t$. It is clear that $(C_t)_{t\geq0}$, $(D_t)_{t\geq0}$ and $(H_t)_{t\geq0}$ are also  strictly stable processes because e.g. $(C_t)_{t\geq0}$ is nothing but $(S_t)_{t\geq0}$ when $a = 1$ and $b = 0$. It is also known from Biane-Yor \cite{biane1987valeurs} that for all $\xi\in\mathbb{R}$, for all $t\geq0$ that

\begin{equation}\label{characterisitic_centered}
    \mathbb{E}\left[\exp(i\xi H_t)\right] = \exp\left(-tc_{\alpha}|\xi|^{\alpha}\right) \:\:\:\text{with}\:\:\:c_{\alpha} = \frac{2^{\alpha - 1}\pi}{\alpha\sin(\alpha\pi/2)}\left(\frac{\alpha^{\alpha}}{\Gamma(\alpha)}\right)^2.
\end{equation}

\noindent
Using the occupation times formula, we can write

$$C_t = \int_0^\infty|x|^{1/\alpha - 2}L_{\tau_t}^x \dr x \quad \text{and}\quad D_t = \int_{-\infty}^0|x|^{1/\alpha - 2}L_{\tau_t}^x \dr x.$$

\noindent
But it is well known, see \cite[page 215]{jeulin1981distributions}, that for a fixed $t$, the processes $(L_{\tau_t}^x)_{x\geq0}$ and $(L_{\tau_t}^x)_{x\leq0}$ are independent and thus $C_t$ and $D_t$ are independent. It is also clear that they have the same law. Since $(C_t)_{t\geq0}$ has only positive jumps, Theorem \ref{stable_def_thm} tells us that its characteristic function can be written

$$\mathbb{E}\left[\exp(i\xi C_t)\right] = \exp\left(-c_{\alpha}'t|\xi|^{\alpha}\left(1-i\tan\left(\frac{\alpha\pi}{2}\right) \mathrm{sgn}(\xi)\right)\right),$$

\noindent
for some $c_{\alpha}'>0$. Now since $\mathbb{E}\left[\exp(i\xi H_t)\right] = \mathbb{E}\left[\exp(i\xi C_t)\right]\mathbb{E}\left[\exp(-i\xi D_t)\right] = \exp\left(-2c_{\alpha}'t|\xi|^{\alpha}\right)$ it is clear from \eqref{characterisitic_centered} that $c_{\alpha}' = c_{\alpha} / 2$. Now having in mind that $S_t = aC_t + bD_t$, and $C_t$ and $D_t$ are i.i.d variables, we easily complete the proof.
\end{proof}

\begin{proof}[Proof of Lemma \ref{stable1_2}]

To show that $(S_t)_{t\geq0}$ is a L\'evy process, we first introduce, for $\eta > 0$,
\begin{equation*}
    \begin{split}
        K_t^{\eta} & =\int_0^t \mathrm{sgn}_{a,b}(W_s)|W_s|^{1/\alpha -2}\bm{1}_{\{|W_s| \geq \eta\}}\dr s - \left(\int_{\mathbb{R}}\mathrm{sgn}_{a,b}(x)|x|^{1/\alpha -2}\bm{1}_{\{|x| \geq \eta\}}\dr x\right) L_t^0\\
          & = \int_{\mathbb{R}}\mathrm{sgn}_{a,b}(x)|x|^{1/\alpha -2}\bm{1}_{\{|x| \geq \eta\}}(L_t^x -L_t^0)\dr x,
    \end{split}
\end{equation*}

\noindent
where we used the occupation times formula. We show that $(K_t^{\eta})_{t\geq 0}$ converges a.s., uniformly on compact time intervals to $(K_t)_{t\geq 0}$. For $\eta > 0$, we have 
\begin{equation*}
    \begin{split}
        \sup_{t\in[0,T]}\left|K_t^{\eta} - K_t\right| & \leq \int_{\mathbb{R}} \left|\mathrm{sgn}_{a,b}(x)|x|^{1/\alpha -2}\bm{1}_{\{|x| \geq \eta\}} - \mathrm{sgn}_{a,b}(x)|x|^{1/\alpha -2}\right|\times\sup_{t\in[0,T]}\left|L_t^x - L_t^0\right|\dr x\\
          & \leq \left(\sup_{(t,x)\in[0, T]\times\mathbb{R}}|x|^{-\theta}\left|L_t^x -L_t^0\right|\right) \int_{\mathbb{R}} \bm{1}_{\{|x| < \eta\}} |x|^{\theta + 1/\alpha - 2}\dr x,
    \end{split}
\end{equation*}

\noindent
where we have fixed $\theta \in (0, 1/2)$ such that $1/\alpha - 2 +\theta > -1$, which is possible since $\alpha <2$. We conclude using that $\sup_{(t,x)\in[0, T]\times\mathbb{R}}|x|^{-\theta}\left|L_t^x -L_t^0\right| < \infty$ a.s., see \eqref{holder_local_times}.

\medskip
From Remark \ref{levy}, and since a.s. for all $t\geq0$, $L_{\tau_t}^0 = t$, we know that $(S_t^{\eta})_{t\geq0} = (K^{\eta}_{\tau_t})_{t\geq0}$ is a L\'evy process for all $\eta>0$. Hence $(S_t)_{t\geq0}$ is a L\'evy process: for all $t\geq0$, a.s., $S_t = \lim_{\eta\to0}S_t^{\eta}$ so that for all $t_1, \dots, t_n \geq0$, a.s., $(S_{t_1}, \dots, S_{t_n}) = \lim_{\eta\to0}(S_{t_1}^{\eta}, \dots, S_{t_n}^{\eta})$ and thus $(S_t)_{t\geq0}$ is the limit of $(S_t^{\eta})_{t\geq0}$ in the finite dimensional distribution sense, which is sufficient to conclude. We now show that it has the appropriate scaling property. By \eqref{scaling}, we have, for any $c > 0$,
\begin{equation*}
    \begin{split}
        S_t & = \int_{\mathbb{R}}\mathrm{sgn}_{a,b}(x)|x|^{1/\alpha -2}(L_{\tau_t}^x -t)\dr x \\  & \overset{d}{=} \int_{\mathbb{R}}\mathrm{sgn}_{a,b}(x)|x|^{1/\alpha -2}(c^{-1}L_{\tau_{ct}}^{cx} -t)\dr x\\
          & = c^{-1/\alpha}\int_{\mathbb{R}}\mathrm{sgn}_{a,b}(x)|x|^{1/\alpha -2}(L_{\tau_{ct}}^{x} -ct)\dr x,
    \end{split}
\end{equation*}

\noindent
which equals $c^{-1/\alpha}S_{ct}$.

\medskip
As in the previous proof, we consider the processes

$$C_t = \int_{0}^{\infty}|x|^{1/\alpha -2}(L_{\tau_t}^x -t)\dr x, \qquad D_t =\int_{-\infty}^{0}|x|^{1/\alpha -2}(L_{\tau_t}^x -t)\dr x$$

\noindent
and $H_t = C_t - D_t$. For a fixed $t\geq 0$, $C_t$ and $D_t$ are independent and identically distributed random variables and it is clear that they only have positive jumps and thus their characteristic function is equal to $\exp\left(-c_{\alpha}'t|\xi|^{\alpha}\left(1-i\tan\left(\frac{\alpha\pi}{2}\right) \mathrm{sgn}(\xi)\right)\right)$ for some $c_{\alpha}' >0$ by Theorem \ref{stable_def_thm}. But once again we know from Biane-Yor \cite{biane1987valeurs} that the characteristic function of $H_t$ is equal to $\exp\left(-tc_{\alpha}|\xi|^{\alpha}\right)$ so once again it comes that $c_{\alpha}' = c_{\alpha}/2$ and we can conlude as in the previous proof.
\end{proof}

\begin{proof}[Proof of Lemma \ref{stable_1}]
As previously, we can show that $(S_t)_{t\geq 0}$ is a L\'evy process by approximation. Then it is enough to show that the characteristic function of $S_t$ has the stated form. To this end, we approach $(S_t)_{t\geq 0}$ by the $\alpha$-stable processes from Lemma \ref{stable1_2}, when $\alpha\in(1,2)$. We denote by $(K_t^{\alpha})_{t\geq 0}$ and $(S_t^{\alpha})_{t\geq 0}$ the processes from Lemma \ref{stable1_2}. We first show that for all $T >0$, a.s.,
\begin{equation}\label{aim_stable_1}
    \sup_{t\in[0,T]}\left|K_t^{\alpha} + \left(\int_{|x|\geq1}\mathrm{sgn}_{a,b}(x)|x|^{1/\alpha -2}\dr x\right)L_t^0 - K_t\right| \underset{\alpha\downarrow 1}{\longrightarrow} 0.
\end{equation}
\noindent
Since $\int_{|x|\geq1}\mathrm{sgn}_{a,b}(x)|x|^{1/\alpha -2}\dr x = \frac{(a+b)\alpha}{\alpha-1}$ and $1/\alpha - 2 < -1$ for $\alpha\in(1,2)$, we can write
\begin{align*}
    K_t^{\alpha} + \frac{(a+b)\alpha}{\alpha-1}L_t^0 - K_t = & \int_{|x|\leq1}\mathrm{sgn}_{a,b}(x)\left(|x|^{1/\alpha -2} - |x|^{-1}\right)\left(L_t^x - L_t^0\right)\dr x \\
     & + \int_{|x|>1}\mathrm{sgn}_{a,b}(x)\left(|x|^{1/\alpha -2} - |x|^{-1}\right)L_t^x \dr x.
\end{align*}
\noindent
Let us introduce $M_t = \sup_{s\in[0,t]}|W_s|$. Then a.s., for any $x>M_t$, $L_t^x = 0$. Set $c = |a|\vee|b|$, for any $T >0$, a.s.,
\begin{align*}
    \sup_{t\in[0,T]}\left|K_t^{\alpha} + \frac{(a+b)\alpha}{\alpha-1}L_t^0 - K_t\right| \leq & c\int_{|x|\leq1}\left(|x|^{1/\alpha -2} - |x|^{-1}\right)\sup_{t\in[0,T]}\left|L_t^x - L_t^0\right| \dr x \\
     & + c\bm{1}_{\{M_T\geq1\}}\int_{1\leq x\leq M_T}\left(|x|^{-1} - |x|^{1/\alpha -2}\right)L_T^x \dr x.
\end{align*}
\noindent
Let us call $R_T^{1,\alpha}$ the first term on the right-hand-side of the above inequality and $R_T^{2,\alpha}$ the second. Then we have for any $T>0$, a.s.,
$$R_T^{1,\alpha} \leq c\left(\sup_{(t,x)\in[0, T]\times[-1,1]}|x|^{-1/3}\left|L_t^x -L_t^0\right|\right)\int_{|x|\leq1}\left(|x|^{1/\alpha -2} - |x|^{-1}\right)|x|^{1/3}\dr x.$$
\noindent
By \eqref{holder_local_times} again, the quantity $\sup_{(t,x)\in[0, T]\times[-1,1]}|x|^{-1/3}\left|L_t^x -L_t^0\right|$ is a.s. finite. For any $x\neq0$, the integrand in the above quantity converges to $0$ as $\alpha$ decreases to $1$. We can use dominated convergence since for any $\alpha\in[1, 4/3)$, for any $x\in[-1, 1]\setminus\{0\}$, we have $(|x|^{1/\alpha -2} - |x|^{-1})|x|^{1/3} \leq (|x|^{-5/4} - |x|^{-1})|x|^{1/3}$, which is integrable on $(-1,0)$ and $(0,1)$. We conclude that $R_T^{1,\alpha}$ converges to $0$ almost surely as $\alpha\downarrow 1$. Regarding $R_T^{2,\alpha}$, we have a.s., for any $T>0$, a.s.
\begin{align*}
    R_T^{2,\alpha} & \leq c \bm{1}_{\{M_T\geq1\}}\sup_{(t,x)\in[0,T]\times \mathbb{R}}|L_t^x|\times\int_{1\leq x\leq M_T}\left(|x|^{-1} - |x|^{1/\alpha -2}\right) \dr x \\
     & \leq c\bm{1}_{\{M_T\geq1\}}\sup_{(t,x)\in[0,T]\times \mathbb{R}}|L_t^x|\times\int_{1\leq x\leq M_T}|x|^{-1}\dr x.
\end{align*}
\noindent
Again, we conclude by dominated convergence that $R_T^{2,\alpha}$ converges to $0$ almost surely as $\alpha\downarrow 1$. All in all, we showed that \eqref{aim_stable_1} holds.

\bigskip
As a consequence, for all $t\geq0$ and all $\xi\in\mathbb{R}$, $$\lim_{\alpha\downarrow 1}\mathbb{E}\left[\exp\left(i\xi S_t^{\alpha}\right)\right]e^{i\xi t\frac{(a+b)\alpha}{\alpha-1}} = \mathbb{E}\left[\exp\left(i\xi S_t\right)\right].$$ 

\noindent
But we know from Lemma \ref{stable1_2} that
$$\mathbb{E}\left[\exp\left(i\xi S_t^{\alpha}\right)\right]e^{i\xi t\frac{(a+b)\alpha}{\alpha-1}} = \exp\left(-c_{\alpha,a,b}t|\xi|^{\alpha} + it\xi\left(c_{\alpha,a,b}\beta_{\alpha,a,b}|\xi|^{\alpha-1}\tan\left(\frac{\alpha\pi}{2}\right) + \frac{(a+b)\alpha}{\alpha - 1}\right)\right),$$

\noindent
where $c_{\alpha,a,b}\beta_{\alpha,a,b} = \frac{2^{\alpha - 2}\pi}{\alpha\sin(\alpha\pi/2)}\left(\frac{\alpha^{\alpha}}{\Gamma(\alpha)}\right)^2(\mathrm{sgn}(a)|a|^{\alpha} + \mathrm{sgn}(b)|b|^{\alpha}) \underset{\alpha\downarrow 1}{\to}\frac{\pi}{2}(a+b)$. The function $h_{a,b}(\alpha) = c_{\alpha,a,b}\beta_{\alpha,a,b}$ defined for $\alpha > 0$ is $C^1$ in a neighborhood of $1$ and a careful computation shows that
$$h_{a,b}'(1) = \frac{\pi}{2}(a+b)\left[1 + \log(2) + 2\gamma + \frac{\log|a|a + \log|b|b }{a+b}\right],$$

\noindent
where $\gamma$ is the Euler constant, which appears from the fact that $\Gamma'(1) = - \gamma$. Hence we have $c_{\alpha,a,b}\beta_{\alpha,a,b} \underset{\alpha\downarrow 1}{=} \frac{\pi}{2}(a+b) + h_{a,b}'(1)(\alpha - 1) + o(\alpha - 1)$. Now using that $|\xi|^{\alpha-1}  \underset{\alpha\downarrow 1}{=} 1 + (\alpha -1)\log|\xi| + o(\alpha - 1)$ and $\tan\left(\frac{\alpha\pi}{2}\right)\underset{\alpha\downarrow 1}{=} \frac{-2}{\pi(\alpha - 1)} + o(1)$, it comes that
$$c_{\alpha,a,b}\beta_{\alpha,a,b}|\xi|^{\alpha-1}\tan\left(\frac{\alpha\pi}{2}\right) + \frac{(a+b)\alpha}{\alpha - 1} = (a+b) - \frac{2}{\pi}h_{a,b}'(1) - (a+b)\log|\xi| + o(1),$$

\noindent
and thus we have
$$\mathbb{E}\left[\exp\left(i\xi S_t^{\alpha}\right)\right]e^{i\xi t\frac{\alpha}{\alpha-1}}\underset{\alpha\downarrow 1}{\to}\exp\left(-c_{a,b}t|\xi|\left(1+i\beta_{a, b}\frac{2}{\pi}\log(|\xi|) \mathrm{sgn}(\xi)\right) + it\tau_{a,b}\xi\right),$$

\noindent
where $c_{a,b} = \frac{\pi}{2} (|a| + |b|)$, $\beta_{a, b} = \frac{a + b}{|a| + |b|}$ and $\tau_{a,b} = -(a + b)\left[2\gamma + \log(2) + \frac{\log|a|a + \log|b|b }{a+b}\right]$.
\end{proof}

\subsection{Inverting time-changes}

We remind some classical convergence results enabling to treat the convergence of the generalized inverses of time changes.

\begin{lemma}\label{inverse_time}
For all $n\geq 1$, let $(a_t^n)_{t\geq0}$ from $[0,+\infty)$ to itself be a continuous, increasing and bijective function for all $n\geq 1$ and consider its inverse $(r_t^n)_{t\geq0}$. Assume $(a_t^n)_{t\geq0}$ simply converges, as $n\to\infty$, to a non-decreasing function $(a_t)_{t\geq0}$ such that $\lim_{t\to+\infty}a_t = \infty$. Consider $r_t = \inf \{u\geq0, a_u > t\}$ its right-continuous generalized inverse and $J = \{t\geq0, r_{t_-} < r_t\}$. For all $t\in [0, +\infty)\setminus J$, we have $\lim_{n\to+\infty}r_t^n = r_t$.
\end{lemma}

\noindent
This lemma is classical and we took the above statement in Fournier-Tardif \cite[Lemma 8]{fournier2018one}. Actually, the proof of this lemma is not different from that of the pointwise convergence of inverse distribution functions, used to show \textit{Skorokhod's representation theorem}, see for instance Billinglsey \cite[Chapter 5, Theorem 25.6]{proba_and_measure}.

\section{Main proofs}\label{proof}

From now on, we will always suppose at least that Assumptions \ref{assump_b_sigma} and \ref{assump_f} holds. In Subsection \ref{section_integrability}, we characterize the integrability of $f$ and $g'\sigma$. In Subsection \ref{scalefunction}, we represent the process $(X_{t/\epsilon})_{t\geq0}$ with a time-changed Brownian motion $(W_{\tau_t^{\epsilon}})_{t\geq0}$ which enables us to state that $\int_0^{t/\epsilon}f(X_s)\dr s$ is equal in law to $H_{\tau_t^{\epsilon}} = \int_0^{\tau_t^{\epsilon}}\phi_{\epsilon}(W_s)\dr s$.

\medskip
In Subsection \ref{levyregime}, we first check that $A_t^{\epsilon}$, inverse of $\tau_t^{\epsilon}$, converges to the local time at $0$ of $(W_t)_{t\geq0}$ so that $\tau_t^{\epsilon}$ converges towards $\tau_t$. We also slightly prepare the proof of the diffusive and critical diffusive regime. Then, as $\epsilon^{1/\alpha}\ell(1/\epsilon)\phi_{\epsilon}(x)\to\mathrm{sgn}_{f_+,f_-}(x)|x|^{1/\alpha - 2}$, we show that $H_{\tau_t^{\epsilon}}$ converges to $K_{\tau_t}$ a.s. for fixed a fixed $t$, which is enough to conclude. We will work a little bit more in the critical L\'evy regime.

\medskip
In Subsection \ref{diffusive_regime}, we adopt the martingale strategy writing $\int_0^{t/\epsilon}f(X_s)\dr s$ as a local martingale plus some remainder and we use classical central limit theorems for local martingales to conclude.

\subsection{Integrability of $f$ and $g'\sigma$}\label{section_integrability}

In this subsection, we characterize with the index $\alpha$ the integrability of $f$ and $g'\sigma$ with respect to the measure $\mu$. We first introduce the functions $\psi:\mathbb{R}\to(0,\infty)$ and $\phi:\mathbb{R}\to\mathbb{R}$ defined by
\begin{equation}\label{definition_phi_psi}
    \psi = \left(\sca'\circ \sca^{-1}\right)\times\left(\sigma\circ \sca^{-1}\right)\quad\text{and}\quad\phi = (f\circ \sca^{-1}) / \psi^2.
\end{equation}

\noindent
Assumption \ref{assump_f} precisely tells us that $\phi$ is a regular varying with index $1/\alpha - 2$. Indeed, since $\sca$ is an increasing bijection by Assumption \ref{assump_b_sigma} and Remark \ref{equivalence_harris}, we have
\begin{equation}\label{phi_asymp}
    |x|^{2 - 1/\alpha}\ell(|x|)\phi(x) \underset{x\to\pm\infty}{\longrightarrow}f_{\pm}.
\end{equation} 

\noindent
We make the following proposition and let us insist on the fact that from now on, we work under Assumptions \ref{assump_b_sigma} and \ref{assump_f}.

\begin{proposition}\label{integrability_f}
$f\in \mathrm{\bm{L}}^1(\mu)$ is equivalent to $\alpha > 1$ or $\alpha = 1$ and $\int_1^{\infty}\frac{\dr x}{x\ell(x)} < \infty$.
\end{proposition}

\begin{proof}
If $\alpha>1$, then $2-1/\alpha > 1$. Since $\ell$ is a slowly varying function, we can find $\delta >0$ with $2-1/\alpha - \delta > 1$ such that $|x|^{\delta}\ell(|x|) \to \infty$ as $|x|\to\infty$. Thus we have by \eqref{phi_asymp} that
$$|\phi(x)| \underset{x\to\pm\infty}{=}o\left(|x|^{\delta + 1/\alpha - 2}\right)$$

\noindent
and the function $\phi$ is integrable with respect to the Lebesgue measure on $\mathbb{R}$. But using the substitution $x = \sca(v)$, we have
\begin{equation}\label{phitof}
    \int_{\mathbb{R}}|\phi(x)|\dr x = \int_{\mathbb{R}}\frac{|f(v)|}{\sigma^2(v)\sca'(v)}\dr v = \frac{1}{\kappa}\int_{\mathbb{R}}|f|\dr\mu < \infty.
\end{equation}

\noindent
If $\alpha = 1$ and $\int_1^{\infty}\frac{\dr x}{x\ell(x)} < \infty$, then \eqref{phi_asymp} tells us that $|x|\ell(|x|)\phi(x) \to f_{\pm}$, and thus $\phi$ is integrable with respect to the Lebesgue measure on $\mathbb{R}$ which tells us by \eqref{phitof} that $f$ is integrable with respect to the Lebesgue measure.

We can make a similar reasoning to show that $\alpha < 1$ or $\alpha = 1$ and $\int_1^{\infty}\frac{\dr x}{x\ell(x)} = \infty$ implies that $f\notin\mathrm{\bm{L}}^1(\mu)$.
\end{proof}

\noindent
We now focus on the function $g$ defined by \eqref{solution_Poisson} when $\alpha \geq 2$ and recall that
$$g(x) = 2\int_0^x \sca'(v)\int_v^{\infty}f(u)[\sigma^{2}(u)\sca'(u)]^{-1}\dr u\dr v = 2\kappa^{-1}\int_0^x \sca'(v)\mu(f\bm{1}_{(v,\infty)})\dr v.$$

\noindent
By Assumption \ref{assump_f}, Proposition \ref{integrability_f} and since $\alpha \geq 2$, we have $\mu(f) = 0$ and thus $\mu(f\bm{1}_{(v,\infty)}) = -\mu(f\bm{1}_{(-\infty,v)})$ so that
\begin{equation}\label{g_moins_infty}
    g(x) = -2\int_0^x \sca'(v)\int_{-\infty}^{v}f(u)[\sigma^{2}(u)\sca'(u)]^{-1}\dr u\dr v.
\end{equation}

\noindent
We now express the function $g\circ\sca^{-1}$, using the substitutions $v=\sca^{-1}(y)$ and $u = \sca^{-1}(z)$:
\begin{equation}\label{g_s-1}
    g(\sca^{-1}(x)) = 2\int_0^x \int_y^{\infty}f(\sca^{-1}(z))[\sigma(\sca^{-1}(z))\sca'(\sca^{-1}(z))]^{-2}\dr z\dr y = 2\int_0^x \int_y^{\infty}\phi(z)\dr z \dr y.
\end{equation}

\noindent
Using \eqref{g_moins_infty}, we also get
\begin{equation}\label{g_s-1_-inf}
    g(\sca^{-1}(x)) = -2\int_0^x \int_{-\infty}^{y}\phi(z)\dr z\dr y.
\end{equation}

\noindent
We are now ready to state the proposition characterizing the integrability of $g'\sigma$ with respect to $\mu$.

\begin{proposition}\label{g'sigma}
If $\alpha > 2$, or $\alpha = 2$ and $\rho =  \int_1^{\infty}\left(\int_x^{\infty}\frac{\dr v}{v^{3/2}\ell(v)}\right)^2\dr x < \infty$, then $g'\sigma\in \mathrm{\bm{L}}^{2}(\mu)$.
\end{proposition}

\begin{proof}
From \eqref{g_s-1} and \eqref{g_s-1_-inf}, we get that
\begin{align}\label{g's_-1}
    (g\circ \sca^{-1})'(x) & = 2\int_x^{\infty}\phi(v)\dr v = -2\int_{-\infty}^{x}\phi(v)\dr v.
\end{align}

\noindent
Assume $\alpha > 2$, then $2 - 1/\alpha > 3/2$. From Assumption \ref{assump_f} and since $\ell$ is a slowly varying function, we can find $\delta>0$ such that $2 - 1/\alpha - \delta> 3/2$ and
$$|\phi(x)| \underset{x\to\pm\infty}{=}o\left(|x|^{\delta + 1/\alpha - 2}\right),$$

\noindent
and thus
$$|(g\circ\sca^{-1})'(x)|^2\underset{x\to\pm\infty}{=}o\left(|x|^{2(\delta + 1/\alpha - 1)}\right).$$

\noindent
Since $2(\delta + 1/\alpha - 1) < -1$, $(g\circ \sca^{-1})' \in \mathrm{\bm{L}}^{2}(\dr x)$. Now using the fact that $(g\circ \sca^{-1})' = [\sca'\circ \sca^{-1}]^{-1}\times g'\circ \sca^{-1}$ and the substitution $x = \sca(v)$, we have
$$\int_{\mathbb{R}}\frac{|g'(\sca^{-1}(x))|^{2}}{[\sca'(\sca^{-1}(x))]^{2}}\dr x = \int_{\mathbb{R}}\frac{|g'(v)|^{2}}{\sca'(v)}\dr v = \frac{1}{\kappa}\int_{\mathbb{R}}|g'\sigma|^{2}\dr\mu,$$

\noindent
whence the result. Now assume $\alpha = 2$ and $\rho = \int_1^{\infty}\left(\int_x^{\infty}\frac{\dr v}{v^{3/2}\ell(v)}\right)^2\dr x < \infty$. Assumption \ref{assump_f} tells us that $\phi(x) \underset{x\to\infty}{\sim}f_{\pm}\frac{1}{|x|^{3/2}\ell(|x|)}$. Integrating and using \eqref{g's_-1}, we get
$$(g\circ\sca^{-1})'(x)\underset{x\to\infty}{\sim}2f_{+}\int_x^{\infty}\frac{\dr v}{v^{3/2}\ell(v)}.$$

\noindent
Hence the fact that $\rho < \infty$ tells us precisely that $(g\circ \sca^{-1})' \in \mathrm{\bm{L}}^{2}(\mathbb{R}_+, \dr x)$. Using a similar computation involving \eqref{g's_-1}, we find that $(g\circ \sca^{-1})' \in \mathrm{\bm{L}}^{2}(\mathbb{R}_-, \dr x)$. We conclude that $g'\sigma\in\mathrm{\bm{L}}^{2}(\mu)$ as when $\alpha > 2$.
\end{proof}

\subsection{Scale function and speed measure}\label{scalefunction}

In this subsection, we classically represent the solution $(X_t)_{t\geq0}$ to \eqref{SDE} as a function of a time-changed Brownian motion. Roughly, the process $(\sca(X_t))_{t\geq0}$ is a continuous local martingale and the Dubins-Schwarz theorem tells us $(\sca(X_t))_{t\geq0}$ is a time-changed Brownian motion. The next lemma enables us to represent $X_{t/\epsilon}$ and $\int_0^{t/\epsilon}f(X_s)\dr s$ through the functions $\psi$, $\phi$, defined by \eqref{definition_phi_psi}, and $\sca$, generalizing to our context the identity found in Fournier-Tardif \cite[Lemma 6]{fournier2018one}.

\begin{lemma}\label{first_lemma}
Let $\epsilon >0$ and $a_{\epsilon} >0$. We consider a Brownian motion $(W_t)_{t\geq 0}$ and we define

$$A_t^{\epsilon} = \epsilon a_{\epsilon}^{-2}\int_0^t \psi^{-2}\left(W_s/a_{\epsilon}\right)\dr s$$

\noindent
and its inverse $(\tau_t^{\epsilon})_{t\geq 0}$,
which is a.s. a continuous and strictly increasing bijection from $[0,+\infty)$ in itself. Let us set
$$X_t^{\epsilon} = \sca^{-1}\left(W_{\tau_t^{\epsilon}}/a_{\epsilon}\right)\quad \text{and}\quad F_t^{\epsilon} = H_{\tau_t^{\epsilon}}^{\epsilon} \quad\text{where}\quad H_t^{\epsilon} = a_{\epsilon}^{-2}\int_0^t \phi\left(W_s/a_{\epsilon}\right)\dr s.$$

\noindent
For $(X_t)_{t\geq 0}$ the unique solution to \eqref{SDE}, we have
$$\left(\int_0^{t/\epsilon}f(X_s)\dr s, X_{t/\epsilon}\right)_{t\geq0} \overset{d}{=} \left(F_t^{\epsilon}, X_t^{\epsilon}\right)_{t\geq0}.$$
\end{lemma}

We introduce the unusual degree of freedom $a_{\epsilon}$ that will allow us to make converge the time-change $A_t^{\epsilon}$ without normalization.

\begin{proof}
We recall that the function $\psi$ and $\phi$ are defined by $\psi = \left(\sca'\circ \sca^{-1}\right)\times\left(\sigma\circ \sca^{-1}\right)$ and $\phi = (f\circ \sca^{-1}) / \psi^2$. Let us consider the function $\psi_{\epsilon}(w) = \epsilon^{-1/2} a_{\epsilon}\psi(w/a_{\epsilon})$, so that $A_t^{\epsilon} = \int_0^t \psi_{\epsilon}^{-2}(W_s)\dr s$. We set $Y_t^{\epsilon} = W_{\tau_t^{\epsilon}}$ which classicaly solves, see \cite[page 452]{kallenberg2006foundations},
$$Y_t^{\epsilon} = \int_0^t \psi_{\epsilon} \left(Y_s^{\epsilon}\right) \dr B_s^{\epsilon}$$

\noindent
where $(B_t^{\epsilon})_{t\geq 0}$ is a Brownian motion. Next, we consider the function $\varphi_{\epsilon}(y) = \sca^{-1}(y / a_{\epsilon})$. We can apply Itô's formula to  $X_t^{\epsilon} = \varphi_{\epsilon}(Y_t^{\epsilon})$ since $\sca^{-1}$ is $C^2$ and we get
$$X_t^{\epsilon} = \int_0^t \varphi_{\epsilon}' (Y_s^{\epsilon})\psi_{\epsilon} \left(Y_s^{\epsilon}\right)\dr B_s^{\epsilon} +  \frac{1}{2}\int_0^t \varphi_{\epsilon}'' (Y_s^{\epsilon})\psi_{\epsilon}^2 \left(Y_s^{\epsilon}\right) \dr s.$$

\noindent
The functions are such that:
$$\varphi_{\epsilon}'(y)\psi_{\epsilon} (y) = \frac{1}{a_{\epsilon}\sca'(\sca^{-1}(y / a_{\epsilon}))}\times \frac{a_{\epsilon}\psi(y / a_{\epsilon})}{\epsilon^{1/2}} = \frac{\sigma(\sca^{-1}(y/a_{\epsilon}))}{\epsilon^{1 / 2}} = \frac{\sigma(\varphi_{\epsilon}(y))}{\epsilon^{1/2}} $$

\noindent
and,
$$\varphi_{\epsilon}''(y)\psi_{\epsilon}^2 (y) = \frac{-1}{a_{\epsilon}^2}\frac{\sca''(\sca^{-1}(y / a_{\epsilon}))}{[\sca'(\sca^{-1}(y / a_{\epsilon}))]^3} \times \frac{a_{\epsilon}^2\psi^2(y / a_{\epsilon})}{\epsilon} = \frac{2b(\sca^{-1}(y / a_{\epsilon}))}{\epsilon} = \frac{2b(\varphi_{\epsilon}(y))}{\epsilon}.$$

\noindent
Indeed remember that $\sca$ satisfies $\frac{1}{2}\sca''\sigma^2 + \sca'b = 0$. Finally we have
$$X_t^{\epsilon} = \epsilon^{-1}\int_0^t b(X_s^{\epsilon})\dr s + \epsilon^{-1 / 2}\int_0^t\sigma(X_s^{\epsilon})\dr B_s^{\epsilon} $$

\noindent
Now taking equation \eqref{SDE}, we write 
$$X_{t/\epsilon} = \int_0^{t/\epsilon} b(X_s)\dr s + \int_0^{t/\epsilon}\sigma(X_s)\dr B_s = \epsilon^{-1}\int_0^t b(X_{s/\epsilon})\dr s + \epsilon^{-1 / 2}\int_0^t\sigma(X_{s/\epsilon})\dr \Hat{B}_s^{\epsilon},$$

\noindent
where $(\Hat{B}_t^{\epsilon})_{t\geq0} = (\epsilon^{1 / 2}B_{t/\epsilon})_{t\geq0}$ is a Brownian motion. Thus, $\forall\epsilon>0$, the processes $(X_{t/\epsilon})_{t\geq0}$ and $(X_t^{\epsilon})_{t\geq0}$ are solutions of the same SDE, for which we have uniqueness in law, driven by different Brownian motion, $(\Hat{B}_t^{\epsilon})_{t\geq 0}$ and $(B_t^{\epsilon})_{t\geq 0}$. Thus they are equal in law : $(X_{t/\epsilon})_{t\geq0} \overset{d}{=}(X_t^{\epsilon})_{t\geq0}$. As $\int_0^{t/\epsilon}f(X_s)\dr s = \epsilon^{-1}\int_0^{t}f(X_{s/\epsilon}) \dr s$, we have 
$$\left(\int_0^{t/\epsilon}f(X_s)\dr s, X_{t/\epsilon}\right)_{t\geq0} \overset{d}{=} \left(\epsilon^{-1}\int_0^{t}f(X_{s}^{\epsilon}) \dr s, X_t^{\epsilon}\right)_{t\geq0}.$$ 

\noindent
Using the substitution $u = \tau_s^{\epsilon} \Leftrightarrow s = A_u^{\epsilon}$, we have $\dr s = \psi_{\epsilon}^{-2}(W_u)
\dr u$, and
$$\epsilon^{-1}\int_0^{t}f(X_{s}^{\epsilon})\dr s = \epsilon^{-1}\int_0^t f\circ\varphi_{\epsilon}\left(W_{\tau_s^{\epsilon}}\right)\dr s = \epsilon^{-1} \int_0^{\tau_t^{\epsilon}}\frac{f\circ\varphi_{\epsilon}(W_u)}{\psi_{\epsilon}^2(W_u)}\dr u = a_{\epsilon}^{-2} \int_0^{\tau_t^{\epsilon}}\phi(W_u / a_{\epsilon})\dr u = F_t^{\epsilon}$$

\noindent
which ends the proof.
\end{proof}

\noindent
We end this section with the following proposition that we will use several times later on.

\begin{proposition}\label{conv_local}
Let $\epsilon>0$ and $a_{\epsilon} > 0$ such that $a_{\epsilon}\to0$ when $\epsilon\to0$. Let also $(W_t)_{t\geq0}$ be a Brownian motion and $(L_t^0)_{t\geq0}$ its local time at $0$. For every $\varphi\in \mathrm{\bm{L}}^1(\mu)$, we have a.s., for all $T>0$,
$$\sup_{t\in [0, T]}\left|\kappa a_{\epsilon}^{-1}\int_0^t \frac{\varphi\circ \sca^{-1}(W_s / a_{\epsilon})}{\psi^2(W_s/a_{\epsilon})}\dr s - \mu(\varphi)L_t^0\right| \underset{\epsilon\to0}{\longrightarrow} 0.$$

\end{proposition}

\begin{proof}
Denote $Y_t^{\epsilon} = \kappa a_{\epsilon}^{-1}\int_0^t \frac{\varphi\circ \sca^{-1}(W_s / a_{\epsilon})}{\psi^2(W_s/a_{\epsilon})}\dr s$. Using the occupation times formula, we get
$$Y_t^{\epsilon} = \kappa a_{\epsilon}^{-1}\int_{\mathbb{R}} \frac{\varphi\circ \sca^{-1}(x / a_{\epsilon})}{\psi^2(x/a_{\epsilon})}L_t^x \dr x =  \kappa\int_{\mathbb{R}}\frac{\varphi(v)\sca'(v)}{\psi^2(\sca(v))}L_t^{a_{\epsilon}\sca(v)}\dr v = \int_{\mathbb{R}}\varphi(v)L_t^{a_{\epsilon}\sca(v)}\mu(\dr v),$$

\noindent
where we used the substitution $x = a_{\epsilon}\sca(v)$. We recall that $\mu(\dr x) = \kappa[\sigma^2(x)\sca'(x)]^{-1}\dr x$ and that $\psi = (\sca'\circ \sca^{-1})\times(\sigma\circ \sca^{-1})$. Hence we have
$$\sup_{t\in [0, T]}\left|Y_t^{\epsilon} - \mu(\varphi)L_t^0\right| \leq \int_{\mathbb{R}}|\varphi(v)|\sup_{t\in [0, T]}\left|L_t^{a_{\epsilon}\sca(v)}-L_t^0\right|\mu(\dr v).$$

\noindent
But for all $v\in\mathbb{R}$, $\sup_{t\in [0, T]}|L_t^{a_{\epsilon}\sca(v)} - L_t^0| \underset{\epsilon\to 0}{\longrightarrow} 0$ a.s. (see \cite{revuz2013continuous} Corollary 1.8, page 226), and $\sup_{t\in[0, T]}|L_t^{a_{\epsilon}\sca(v)} - L_t^0| \leq 2 \sup_{(t,x)\in[0, T]\times\mathbb{R}}L_t^x < \infty$ a.s.. We can conclude by dominated convergence since $\varphi\in\mathrm{\bm{L}}^1(\mu)$ by assumption.
\end{proof}

\subsection{The L\'evy regime}\label{levyregime}

In this section, which generalizes \cite[Lemma 9]{fournier2018one}, we first show two lemmas, that will enable us to conclude. The first lemma says that if in Lemma \ref{first_lemma} we choose carefully $a_{\epsilon}$, the time change $A_t^{\epsilon}$ converges to the local time of the Brownian motion at 0. It also recovers a weak version of the ergodic theorem and prepares some useful results for the diffusive case.

\begin{lemma}\label{lemma_time_change}
Let $(W_t)_{t\geq0}$ be a Brownian motion and $(L_t^0)_{t\geq0}$ its local time at 0. For all $\epsilon > 0$, we consider the processes $(A_t^{\epsilon})_{t\geq0}$, $(\tau_t^{\epsilon})_{t\geq0}$ and $(X_t^{\epsilon})_{t\geq0}$ from Lemma \ref{first_lemma} with the choice $a_{\epsilon} = \epsilon/\kappa$.

\begin{enumerate}[label=(\roman*)]
    \item We have a.s., for all $T > 0$, $$ \sup_{t\in[0, T]}\left|A_t^{\epsilon} - L_t^0\right| \underset{\epsilon\to0}{\longrightarrow} 0.$$
    
    \item For all $t > 0$, a.s., $\tau_t^{\epsilon} \underset{\epsilon\to0}{\longrightarrow} \tau_t$, the generalized inverse of $(L_t^0)_{t\geq 0}$.
    
    \item Suppose Assumption \ref{assump_mu_integ}. For any slowly varying function $\gamma$,  we have a.s., for all $T > 0$, $$ \sup_{t\in[0, T]}\gamma(1/\epsilon)\left|A_t^{\epsilon} - L_t^0\right| \underset{\epsilon\to0}{\longrightarrow} 0.$$
    
    \item If $\alpha = 2$ and $\rho = \int_1^{\infty}\left(\int_x^{\infty}\frac{\dr v}{v^{3/2}\ell(v)}\right)^2\dr x = \infty$, we have a.s., for all $T > 0$,
    $$\sup_{t\in[0, T]}\left|T_t^{\epsilon} - \sigma_{2}^2 L_t^0\right| \underset{\epsilon\to0}{\longrightarrow} 0 \quad\text{where}\quad T_t^{\epsilon} = \kappa^2 |\epsilon\rho_{\epsilon}|^{-1}\int_0^t \chi(W_s / a_{\epsilon})\dr s,$$
    \noindent
    where $\chi = [(g\circ \sca^{-1})']^{2}$ with $g$ defined by \eqref{solution_Poisson} and $\rho_{\epsilon} = \int_1^{1/\epsilon}\left(\int_x^{\infty}\frac{\dr v}{v^{3/2}\ell(v)}\right)^2\dr x$.
    
    \item For all $\varphi\in \mathrm{\bm{L}}^1(\mu)$, for all $t > 0$, we have a.s.
    $$\int_0^t\varphi(X_s^{\epsilon})\dr s \underset{\epsilon\to0}{\longrightarrow} \mu(\varphi)t.$$
\end{enumerate}
\end{lemma}

\begin{proof}
\textbf{Item (i).} The first point is just the application of Proposition \ref{conv_local} with $\varphi = 1$.

\bigskip\noindent
\textbf{Item (ii).} The second point follows immediately from the first and Lemma \ref{inverse_time}, since for all $t\geq0$, as recalled in Subsection \ref{section_local_times}, $\mathbb{P}(\tau_{t-} < \tau_t) = 0$.

\bigskip\noindent
\textbf{Item (iii).} Using the same scheme as in the proof of Proposition \ref{conv_local}, we have 
$$\sup_{t\in[0, T]}\gamma(1/\epsilon)\left|A_t^{\epsilon}\! - \! L_t^0\right| \!= \!\sup_{t\in[0, T]}\left|\kappa\gamma(1/\epsilon)\int_{\mathbb{R}}\frac{L_t^{\epsilon y / \kappa} - L_t^0}{\psi^2(y)}\dr y\right| \leq \int_{\mathbb{R}}\gamma(1/\epsilon)\sup_{t\in[0, T]}\left|L_t^{\epsilon \sca(v) / \kappa} - L_t^0\right|\mu(\dr v).$$

\noindent
Now we use \eqref{holder_local_times} with some $\theta \in (0,\frac{1}{2}\wedge\lambda)$, where $\lambda$ refers to the constant of Assumption \ref{assump_mu_integ}, so that $|\sca|^{\lambda}\in \mathrm{\bm{L}}^1(\mu)$. We thus have $\gamma(1/\epsilon)\sup_{t\in[0, T]}|L_t^{\epsilon \sca(v) / \kappa} - L_t^0| \leq C|\sca(v)|^{\theta}\gamma(1/\epsilon)\epsilon^{\theta}$, for some random $C$. Since $\gamma$ is slowly varying, we conclude that $\gamma(1/\epsilon)\sup_{t\in[0, T]}|L_t^{\epsilon \sca(v) / \kappa} - L_t^0| \longrightarrow 0$ as $\epsilon\to0$, for each fixed $v$. Moreover $$\sup_{\epsilon\in(0,1)}\gamma(1/\epsilon)\sup_{t\in [0, T]}\left|L_t^{\epsilon \sca(v) / \kappa} - L_t^0\right| \leq C\left(\sup_{\epsilon\in(0,1)}\gamma(1/\epsilon)\epsilon^{\theta}\right) |\sca(v)|
^{\theta}$$ 
\noindent
and we can conclude by dominated convergence since $|\sca|^{\theta}\in \mathrm{\bm{L}}^1(\mu)$.

\bigskip\noindent
\textbf{Item (iv).} \textit{Step 0 :} We first estimate asymptotically the function $\chi$. Let us recall from \eqref{g's_-1} that
$$(g\circ \sca^{-1})'(x) = 2\int_x^{\infty}\phi(v)\dr v = -2\int_{-\infty}^{x}\phi(v)\dr v.$$

\noindent
Let us define the function $h$ on $(0,\infty)$:
$$h(x) = \left(\int_x^{\infty}\frac{\dr v}{v^{3/2}\ell(v)}\right)^2.$$

\noindent
By \eqref{phi_asymp}, we have $|x|^{3/2}\ell(|x|)\phi(x)\underset{x\to\pm\infty}{\longrightarrow}f_{\pm}$ and thus
$$\left(\int_x^{\infty}\frac{\dr v}{v^{3/2}\ell(v)}\right)^{-1}\int_x^{\infty}\phi(v)\dr v\underset{x\to\infty}{\longrightarrow}f_{+}$$

\noindent
and
$$\left(\int_{|x|}^{\infty}\frac{\dr v}{v^{3/2}\ell(v)}\right)^{-1}\int_{-\infty}^{x}\phi(v)\dr v\underset{x\to-\infty}{\longrightarrow}f_{-}.$$

\noindent
All in all, recalling that $\chi = [(g\circ \sca^{-1})']^{2}$, we have $[h(|x|)]^{-1}\chi(x)\underset{x\to\pm\infty}{\longrightarrow}4f_{\pm}^2$ from which
\begin{equation}\label{equivalent_chi}
    \int_{-x}^x \chi(z)\dr z \underset{x\to\infty}{\sim}4(f_+^2 + f_-^2)\int_1^x h(v)\dr v.
\end{equation}

\noindent
We used that $\rho = \int_1^{\infty}h(v)\dr v = \infty$ by assumption. Moreover there exists $A>0$ such that for all $x$, we have $\chi(x) \leq Ah(|x|)$. Now we write for $\delta > 0$ fixed, $T_t^{\epsilon} = D_t^{\epsilon,\delta} + E_t^{\epsilon,\delta}$ where
$$D_t^{\epsilon,\delta} = \int_0^t \kappa^2|\epsilon\rho_{\epsilon}|^{-1}\chi(\kappa W_s / \epsilon)\bm{1}_{\{|W_s|>\delta\}}\dr s  \:\:\:\text{    and    }\:\:\: E_t^{\epsilon,\delta} = \int_0^t \kappa^2|\epsilon\rho_{\epsilon}|^{-1}\chi(\kappa W_s / \epsilon)\bm{1}_{\{|W_s|\leq\delta\}}\dr s.$$

\noindent
\textit{Step 1 :} We show that $\sup_{t\in[0,T]}|D_t^{\epsilon,\delta}|\to 0$ as $\epsilon\to 0$. We have $\chi(\kappa W_s/\epsilon) \leq Ah(\kappa |W_s|/\epsilon)$. Now notice that in this case, we have a.s.,
\begin{align*}
    \bm{1}_{\{|W_s|>\delta\}}h(\kappa |W_s|/\epsilon) & = \bm{1}_{\{|W_s|>\delta\}}\left(\int_{\kappa |W_s|/\epsilon}^{\infty}\frac{\dr u}{u^{3/2}\ell( u)}\right)^2 \\
       & = \bm{1}_{\{|W_s|>\delta\}}\frac{\epsilon}{\kappa |W_s|}\left(\int_{1}^{\infty}\frac{\dr v}{v^{3/2}\ell(\kappa |W_s| v/\epsilon)}\right)^2 \\
      & \leq \bm{1}_{\{|W_s|>\delta\}}\frac{\epsilon}{\kappa \delta}\left(\int_{1}^{\infty}\frac{\dr v}{v^{3/2}\ell(\kappa |W_s|v/\epsilon)}\right)^2.
\end{align*}

\noindent
We know from Lemma \ref{prop_slow_var}-(i) (with $a = \delta$ and $b = \sup_{s\in [0,T]}|W_s|$) that the following uniform convergence holds a.s.
$$\sup_{s\in [0,T]}\bm{1}_{\{|W_s|>\delta\}}\left|\frac{\ell(v/\epsilon)}{\ell(\kappa |W_s|v/\epsilon)} - 1\right| \underset{\epsilon\to0}{\longrightarrow}0.$$

\noindent
Hence there exists a random $M > 0$ such that for any $s\in[0,T]$ and $\epsilon \in (0,1)$, $|W_s| > \delta$ implies that $1/\ell(\kappa |W_s|v/\epsilon) \leq M/\ell(v/\epsilon)$. Consequently, $$\sup_{s\in[0,T]}\bm{1}_{\{|W_s|>\delta\}}|\chi(\kappa W_s/\epsilon)| \leq \sup_{s\in[0,T]}\bm{1}_{\{|W_s|>\delta\}}Ah(\kappa |W_s|/\epsilon)\leq \frac{\epsilon A M^2}{\kappa\delta}\left(\int_{1}^{\infty}\frac{\dr v}{v^{3/2}\ell(v/\epsilon)}\right)^2.$$     

\noindent
We know from Potter's bound (Lemma \ref{prop_slow_var}-(ii)) that there exists $x_0>1$ such that for all $x, y \geq x_0$,
$$\frac{\ell(y)}{\ell(x)}\leq 2\left(\Big|\frac{y}{x}\Big|^{1/4} \vee \Big|\frac{x}{y}\Big|^{1/4}\right).$$

\noindent
Hence, if $\epsilon\in(0,1/x_0)$, we have
$$\left(\int_{1}^{\infty}\frac{\dr v}{v^{3/2}\ell(v/\epsilon)}\right)^2 \leq \frac{4}{\ell^2(1/\epsilon)}\left(\int_{1}^{\infty}\frac{|v|^{1/4} \vee |v|^{-1/4}\dr v}{v^{3/2}}\right)^2 .$$

\noindent
All in all, there exists a random constant that we name $M$ again, such that
$$\sup_{t\in [0,T]}|D_t^{\epsilon,\delta}| \leq \frac{MT}{\delta\ell^2(1/\epsilon)\rho_{\epsilon}}.$$

\noindent
It holds that $\ell^2(1/\epsilon) \rho_{\epsilon} \longrightarrow \infty$, see the end of the proof of Lemma \ref{prop_slow_var}-(iii) and use that $\rho_{\epsilon} = M(1/\epsilon)$, and thus for each $\delta > 0$, $\sup_{t\in[0,T]}|D_t^{\epsilon,\delta}| \longrightarrow 0$ as $\epsilon \to 0$.

\medskip\noindent
\textit{Step 2 :} We concentrate on $E_t^{\epsilon,\delta}$. Using the occupation times formula, we have
$$E_t^{\epsilon,\delta} = \left(\int_{-\delta}^{\delta}\frac{\kappa ^2\chi(\kappa x/\epsilon) \dr x}{|\epsilon\rho_{\epsilon}|}\right)L_t^0 + \int_{-\delta}^{\delta}\frac{\kappa ^2\chi(\kappa x/\epsilon)(L_t^x - L_t^0) \dr x}{|\epsilon\rho_{\epsilon}|} = r_{\epsilon,\delta}L_t^0 + R_t^{\epsilon,\delta},$$

\noindent
the last equality standing for a definition. We have, using a substitution and \eqref{equivalent_chi},
$$r_{\epsilon,\delta} = \frac{\kappa }{\rho_{\epsilon}}\int_{-\kappa \delta/\epsilon}^{\kappa \delta/\epsilon}\chi(z)\dr z \underset{\epsilon\to0}{\sim} \frac{4\kappa(f_+^2 + f_-^2)}{\rho_{\epsilon}} \int_1^{\kappa \delta/\epsilon} h(x)\dr x \underset{\epsilon\to0}{\sim}4\kappa (f_+^2 + f_-^2) = \sigma_2^2,$$

\noindent
where we used that $\rho_{\epsilon} = \int_1^{1/\epsilon} h(x)\dr x$ and that the function $M(x)=\int_1^x h(v)\dr v$ is slowly varying by Lemma \ref{prop_slow_var}-(iii). 

\medskip\noindent
\textit{Step 3 : }We conclude. Recalling that $T_t^{\epsilon} = r_{\epsilon,\delta}L_t^0 + R_t^{\epsilon,\delta} + D_t^{\epsilon,\delta}$, we proved that a.s., for all $T > 0$ and all $\delta >0$,
$$\limsup_{\epsilon\to0}\sup_{t\in [0,T]}\left|T_t^{\epsilon} - \sigma_2^2 L_t^0\right| \leq \limsup_{\epsilon\to0} \sup_{t\in[0,T]}|R_t^{\epsilon,\delta}|.$$

\noindent
But
$$\sup_{t\in[0,T]}|R_t^{\epsilon,\delta}| \leq r_{\epsilon,\delta}\sup_{(t,x)\in[0,T]\times [-\delta,\delta]}|L_t^x - L_t^0|,$$

\noindent
which implies
$$\limsup_{\epsilon\to0}\sup_{t\in[0,T]}\left|T_t^{\epsilon} - \sigma_2^2 L_t^0\right| \leq \sigma_2^2\sup_{(t,x)\in[0,T]\times [-\delta,\delta]}|L_t^x - L_t^0|.$$

\noindent
Now, we let $\delta$ to $0$ which completes the proof.

\bigskip\noindent
\textbf{Item (v).} Using the definition of $X_t^{\epsilon}$, and making the substitution $u = \tau_s^{\epsilon} \Leftrightarrow s = A_u^{\epsilon}$, we have 
$$\int_0^t\varphi(X_s^{\epsilon})\dr s = \epsilon a_{\epsilon}^{-2}\int_0^{\tau_t^{\epsilon}}\frac{\varphi\circ \sca^{-1}(W_s / a_{\epsilon})}{\psi^2(W_s/a_{\epsilon})}\dr s = \kappa a_{\epsilon}^{-1}\int_0^{\tau_t^{\epsilon}}\frac{\varphi\circ \sca^{-1}(W_s / a_{\epsilon})}{\psi^2(W_s/a_{\epsilon})}\dr s.$$

\noindent
Let $T = \sup_{\epsilon\in(0,1)}\tau_t^{\epsilon}$ which is a.s. finite from the second point. We have
\begin{equation*}
    \begin{split}
        \left|\int_0^t\varphi(X_s^{\epsilon})\dr s - \mu(\varphi)t\right|  \leq & \sup_{u\in[0, T]}\left|\kappa a_{\epsilon}^{-1}\int_0^u \frac{\varphi\circ \sca^{-1}(W_s / a_{\epsilon})}{\psi^2(W_s/a_{\epsilon})}\dr s - \mu(\varphi)L_u^0\right|\\
         & + \mu(\varphi)\left|L_{\tau_t^{\epsilon}}^0 - t\right|.
    \end{split}
\end{equation*}

\noindent
The first term on the right hand side goes to $0$ by Proposition \ref{conv_local} and the second goes to $0$ since $\tau^{\epsilon}_t \to\tau_t$ by Point 2, $L_{\tau_t}^0 = t$ a.s., and $(L_t^0)_{t\geq0}$ is continuous.
\end{proof}

Let us state two immediate consequences of the previous lemma.

\begin{remark}\label{ergodic_version}
Point (5) of the preceding Lemma recovers a weak version of the ergodic theorem for the process $(X_t)_{t\geq0}$. Although it is straightforward since the process is Harris recurrent, this provides a self-contained proof of this weak version, which we only need. For all $\varphi\in \mathrm{\bm{L}}^1(\mu)$ and all $t\geq0$, $\epsilon\int_0^{t/\epsilon}\varphi(X_s)\dr s \overset{d}{=}\int_0^{t}\varphi(X_s^{\epsilon})\dr s$, hence for all $t\geq0$
$$\epsilon\int_0^{t/\epsilon}\varphi(X_s)\dr s \overset{\mathbb{P}}{\longrightarrow}\mu(\varphi)t.$$
\end{remark}

\begin{remark}\label{critical_diffusive_remark}
Point (4) prepares the critical diffusive regime. Using arguments similar to those in the proof of point (5), we easily see that if $\alpha = 2$ and $\rho = \infty$, then for every $t\geq0$, a.s., $T_{\tau_t^{\epsilon}}^{\epsilon} \underset{\epsilon\to0}{\longrightarrow}\sigma_{2}^2 t$. Now one can also see that as $|\epsilon/\rho_{\epsilon}|\int_0^{t/\epsilon}[g'(X_s)\sigma(X_s)]^2\dr s \overset{d}{=}T_{\tau_t^{\epsilon}}^{\epsilon}$, we have
$$|\epsilon/\rho_{\epsilon}|\int_0^{t/\epsilon}[g'(X_s)\sigma(X_s)]^2\dr s \overset{\mathbb{P}}{\longrightarrow}\sigma_{2}^2 t.$$
\end{remark}

The next result is the last preliminary before proving the main result in the L\'evy case.

\begin{lemma}\label{lemme2}
Let $(W_t)_{t\geq0}$ be a Brownian motion, and $(L_t^0)_{t\geq0}$ its local time in 0. Let also $(K_t)_{t\geq0}$ be the process defined in section \ref{stable_process}, with $(a,b) = (f_+, f_-)$. For all $\epsilon > 0$, we consider the process $(H_t^{\epsilon})_{t\geq0}$ from Lemma \ref{first_lemma} with the choice $a_{\epsilon} = \epsilon / \kappa $.

\begin{enumerate}[label=(\roman*)]
    \item If $\alpha \neq 1$,  then a.s., for all $T > 0$, $$\sup_{t\in[0, T]} \left|\epsilon^{1/\alpha}\ell(1/\epsilon)H_t^{\epsilon} - \kappa^{1/\alpha} K_t\right| \underset{\epsilon\to0}{\longrightarrow} 0.$$
    
    \item If $\alpha = 1$, there exists $(\xi_{\epsilon})_{\epsilon >0}$ such that, a.s., for all $T > 0$, $$\sup_{t\in[0, T]}\left|\epsilon \ell(1/\epsilon) H_t^{\epsilon} - \xi_{\epsilon}L_t^0 - \kappa K_t\right| \underset{\epsilon\to0}{\longrightarrow} 0,$$
    
    \noindent
    where $\xi_{\epsilon} \underset{\epsilon\to0}{\sim} -\kappa(f_+ + f_-)\ell(1/\epsilon)\int_{1/\epsilon}^{\infty}\frac{\dr x}{x\ell(x)}$ if $f\in \mathrm{\bm{L}}^1(\mu)$ and $\xi_{\epsilon} \underset{\epsilon\to0}{\sim} \kappa(f_+ + f_-)\ell(1/\epsilon)\int_1^{1/\epsilon}\frac{\dr x}{x\ell(x)}$ otherwise.
\end{enumerate}
\end{lemma}

\begin{proof}
\textbf{Technical estimates.} We first show some technical estimates that we will use in this proof.

\medskip
\textit{(i)} For all $x\in\mathbb{R}^*$, it holds (and this is the case for all $\alpha\in(0,2)$) that
\begin{equation}\label{aaa}
    \kappa ^2\epsilon^{1/\alpha - 2}\ell(1/\epsilon)\phi\left(\kappa  x / \epsilon\right) \underset{\epsilon\to0}{\longrightarrow} \kappa^{1/\alpha} \mathrm{sgn}_{a,b}(x)|x|^{1/\alpha - 2}.
\end{equation}

\noindent
Indeed we have by \eqref{phi_asymp}, since $\ell$ is slowly varying and since $(a,b) = (f_+, f_-)$,
$$\forall x > 0, \quad\kappa ^2\epsilon^{1/\alpha - 2}\ell(1/\epsilon)\phi\left(\kappa  x / \epsilon\right) \underset{\epsilon\to0}{\longrightarrow} \kappa^{1/\alpha} ax^{1/\alpha - 2},$$
$$\forall x < 0,\quad \kappa ^2\epsilon^{1/\alpha - 2}\ell(1/\epsilon)\phi\left(\kappa  x / \epsilon\right) \underset{\epsilon\to0}{\longrightarrow} \kappa^{1/\alpha} b|x|^{1/\alpha - 2}.$$

\medskip
\textit{(ii)} For all $\delta > 0$, there exist $A> 0 $ and $\epsilon_0 > 0$ (and this is the case for all $\alpha\in(0,2)$) such that for all $x\in\mathbb{R}$ and all $\epsilon\in(0,\epsilon_0)$,
\begin{equation}\label{bbb}
    \epsilon^{1/\alpha - 2}\ell(1/\epsilon)\phi\left(\kappa  x / \epsilon\right) \leq A (|x|^{1/\alpha - 2}  + 1)(|x|^{\delta} + |x|^{-\delta})
\end{equation}

\noindent
Indeed, we start using \eqref{phi_asymp} again and that $\phi$ and $\ell$ are locally bounded: there exists $A > 0$ such that for all $x\in\mathbb{R}$, $\ell(|x|)|\phi(x)|\leq A (|x|^{1/\alpha - 2} + 1)$. Hence we have
$$\epsilon^{1/\alpha - 2}\ell(1/\epsilon)|\phi\left(\kappa  x/ \epsilon\right)| \leq A (|x|^{1/\alpha - 2} + 1)\frac{\ell(1/\epsilon)}{\ell(\kappa |x| / \epsilon)},$$

\noindent
the value of $A$ being allowed to change. We know from Potter's bound (Lemma \ref{prop_slow_var}-(ii)) that there exists $x_0$ such that for all $x, y \geq x_0$,
$$\frac{\ell(y)}{\ell(x)}\leq 2\left(\Big|\frac{y}{x}\Big|^{\delta} \vee \Big|\frac{x}{y}\Big|^{\delta}\right).$$

\noindent
Thus, for $\epsilon\in(0, 1/x_0)$ we have 
$$\epsilon^{1/\alpha - 2}\ell(1/\epsilon)|\phi\left(\kappa  x / \epsilon\right)|\bm{1}_{\{\kappa |x|\geq\epsilon x_0\}} \leq A (|x|^{1/\alpha - 2}  + 1)(|x|^{\delta}\vee |x|^{-\delta}).$$

\noindent
The function $\ell$ is strictly positive and bounded from below on every compact set and thus there exists $\ell_0 > 0$ such that for all $x\in[0,x_0]$, $\ell(x)\geq \ell_0$. Now if $\kappa |x|<\epsilon x_0$, we have $\ell(\kappa |x| / \epsilon) \geq \ell_0$ and $1 \leq (\kappa |x|/\epsilon x_0)^{-\delta}$. Hence
$$\epsilon^{1/\alpha - 2}\ell(1/\epsilon)|\phi\left(\kappa  x / \epsilon\right)|\bm{1}_{\{\kappa |x|<\epsilon x_0\}} \leq A(|x|^{1/\alpha - 2} + 1) |x|^{-\delta} \sup_{\epsilon\in(0,1)}\epsilon^{\delta}\ell(1/\epsilon).$$

\noindent
We proved \eqref{bbb}.

\bigskip\noindent
\textbf{Item (i) when $\bm{\alpha \in (0,1)}$}. Let us remind that $K_t = \int_0^t\mathrm{sgn}_{a,b}(W_s)|W_s|^{1/\alpha - 2}ds$. We have $\epsilon^{1/\alpha}\ell(1/\epsilon)H_t^{\epsilon} = \kappa ^2 \epsilon^{1/\alpha - 2}\ell(1/\epsilon)\int_0^t \phi\left(\kappa  W_s / \epsilon\right)\dr s$ and
$$\sup_{t\in [0, T]}\left|\epsilon^{1/\alpha}\ell(1/\epsilon)H_t^{\epsilon} - \kappa^{1/\alpha} K_t\right| \leq \int_0^T \left| \kappa ^2\epsilon^{1/\alpha - 2}\ell(1/\epsilon)\phi\left(\kappa  W_s / \epsilon\right) - \kappa^{1/\alpha}\mathrm{sgn}_{a,b}(W_s)|W_s|^{1/\alpha - 2}\right|\dr s$$

\noindent
which tends to 0 a.s. by dominated convergence. Indeed, since $\{t\in[0,T], W_t = 0\}$ is a.s. Lebesgue-null, we have by \eqref{aaa} that a.s., for a.e. $s\in[0,T]$,
$$|\kappa ^2\epsilon^{1/\alpha - 2}\ell(1/\epsilon)\phi\left(\kappa  W_s / \epsilon\right) - \kappa^{1/\alpha}\mathrm{sgn}_{a,b}(W_s)|W_s|^{1/\alpha - 2}| \to 0$$

\noindent
Next we can dominate $\epsilon^{1/\alpha - 2}\ell(1/\epsilon)|\phi(\kappa  W_s / \epsilon)|$ using \eqref{bbb} by
$$A (|W_s|^{1/\alpha - 2}  + 1)(|W_s|^{\delta} + |W_s|^{-\delta})$$

\noindent
which is a.s. integrable on $[0,T]$ if $\delta > 0$ is small enough so that $1/\alpha - 2 - \delta > - 1$.

\bigskip\noindent
\textbf{Item (i) when $\bm{\alpha \in (1,2)}$}. We notice that $\mu(f) = \kappa  \int_{\mathbb{R}}\phi(z)\dr z$, which equals $0$ by Assumption \ref{assump_f}. Indeed,
\begin{equation}\label{ccc}
    \int_{\mathbb{R}}\phi(z)\dr z = \int_{\mathbb{R}} \frac{f(\sca^{-1}(z))}{\sigma^2(\sca^{-1}(z))\sca'(\sca^{-1}(z))^2}\dr z = \int_{\mathbb{R}}\frac{f(v)}{\sigma^2(v)\sca'(v)}\dr v = \frac{\mu(f)}{\kappa }.
\end{equation}

\noindent
We can now use the occupation times formula to write
\begin{equation*}
    \begin{split}
        \epsilon^{1/\alpha}\ell(1/\epsilon)H_t^{\epsilon} & = \kappa ^2 \epsilon^{1/\alpha - 2}\ell(1/\epsilon)\int_0^t \phi\left(\kappa  W_s / \epsilon\right)\dr s \\
         & = \int_{\mathbb{R}} \kappa ^2 \epsilon^{1/\alpha - 2}\ell(1/\epsilon)\phi\left(\kappa  x / \epsilon\right)L_t^x \dr x \\
         & = \int_{\mathbb{R}} \kappa ^2 \epsilon^{1/\alpha - 2}\ell(1/\epsilon)\phi\left(\kappa  x / \epsilon\right)(L_t^x -L_t^0)\dr x,
    \end{split}
\end{equation*}

\noindent
since $\int_{\mathbb{R}}\phi(z)\dr z = 0$ as noted above. Thus, recalling that $K_t = \int_{\mathbb{R}}\mathrm{sgn}_{a,b}(x)|x|^{1/\alpha -2}(L_t^x -L_t^0)\dr x$ in this case, we have
\begin{align*}
    & \sup_{t\in[0, T]}\left|
    \epsilon^{1/\alpha}\ell(1/\epsilon)H_t^{\epsilon} - \kappa^{1/\alpha} K_t\right| \\
    \leq & \int_{\mathbb{R}}\left|\kappa ^2 \epsilon^{1/\alpha - 2}\ell(1/\epsilon)\phi(\kappa  x / \epsilon) - \kappa^{1/\alpha}\mathrm{sgn}_{a,b}(x)|x|^{1/\alpha -2}\right|\sup_{t\in[0, T]}\left|L_t^x -L_t^0\right|\dr x.
\end{align*}    
    
\noindent
We conclude by dominated convergence. First, by \eqref{aaa}, the integrand converges to $0$ for each $x\in\mathbb{R}^*$.

\medskip\noindent
We next use \eqref{bbb} to dominate $R_{\epsilon}(x) = |\kappa ^2 \epsilon^{1/\alpha - 2}\ell(1/\epsilon)\phi(\kappa  x / \epsilon)|\sup_{t\in[0, T]}\left|L_t^x -L_t^0\right|$ by
$$A (|x|^{1/\alpha - 2}  + 1)(|x|^{\delta} + |x|^{-\delta})|\sup_{t\in[0, T]}\left|L_t^x -L_t^0\right| \leq M (|x|^{1/\alpha - 2}  + 1)(|x|^{\delta} + |x|^{-\delta})| \min(1,|x|^{\theta}),
$$
\noindent
for some random constant $M > 0$ and $\theta\in(0,1/2)$ fixed. We used \eqref{holder_local_times} and the fact that $\sup_{(t,x)\in[0, T]\times\mathbb{R}}\left|L_t^x -L_t^0\right| < \infty$ a.s. This bound is integrable on $\mathbb{R}$ if we choose $\delta > 0$ and $\theta\in(0,1/2)$ such that $1/\alpha - 2 +\delta < -1$ and $1/\alpha - 2 + \theta -\delta > -1 $. This is possible because $-3/2 <1/\alpha - 2 < -1$.

\bigskip\noindent
\textbf{Item (ii): Convergence}. Setting $\xi_{\epsilon} = \kappa ^2\epsilon^{-1}\ell(1/\epsilon)\int_{-1}^1\phi\left(\kappa  x/\epsilon\right)\dr x$ and using the occupation times formula, we have
$$\epsilon \ell(1/\epsilon) H_t^{\epsilon} - \xi_{\epsilon}L_t^0 = \int_{|x|\geq 1}\kappa ^2 \epsilon^{-1}\ell(1/\epsilon)\phi\left(\kappa  x/\epsilon\right)L_t^x \dr x + \int_{|x|\leq 1}\kappa ^2 \epsilon^{-1}\ell(1/\epsilon)\phi\left(\kappa  x/\epsilon\right)(L_t^x - L_t^0) \dr x$$

\noindent
Then, remembering that $K_t = \int_{\mathbb{R}}\mathrm{sgn}_{a,b}(x)|x|^{-1}(L_t^x -L_t^0\bm{1}_{\{|x| \leq 1\}})\dr x$, we write
\begin{align*}
    & \sup_{t\in[0, T]}\left|\epsilon H_t^{\epsilon} - \xi_{\epsilon}L_t^0 - \kappa K_t\right| \\
     \leq &\int_{|x| \leq 1}\left|\kappa ^2 \epsilon^{-1}\ell(1/\epsilon)\phi(\kappa  x/\epsilon) - \kappa\mathrm{sgn}_{a,b}(x)|x|^{-1}\right|\sup_{t\in[0, T]}\left|L_t^x -L_t^0\right|\dr x \\
    & + \int_{|x| > 1}\left|\kappa ^2 \epsilon^{-1}\ell(1/\epsilon)\phi(\kappa  x/\epsilon) - \kappa\mathrm{sgn}_{a,b}(x)|x|^{-1}\right| L_T^x \dr x,
\end{align*}

\noindent
Again, we conclude by dominated convergence. By \eqref{aaa}, the integrand in both terms converges to $0$ for each $x\in\mathbb{R}^*$.

\medskip
We first use \eqref{bbb} to dominate $R_{\epsilon}(x) = |\kappa ^2 \epsilon^{-1}\ell(1/\epsilon)\phi(\kappa  x / \epsilon)|\sup_{t\in[0, T]}\left|L_t^x -L_t^0\right|$, for all $|x|\leq 1$, by
$$M (|x|^{-1}  + 1)(|x|^{\delta} + |x|^{-\delta})|\times|x|^{\theta},
$$

\noindent
for some random constant $M > 0$ and $\theta\in(0,1/2)$ fixed. We used \eqref{holder_local_times} and the fact that $\sup_{(t,x)\in[0, T]\times\mathbb{R}}\left|L_t^x -L_t^0\right| < \infty$ a.s. This bound is integrable on $[-1,1]$ if we choose $\delta > 0$ and $\theta\in(0,1/2)$ such that $-1 + \theta -\delta > -1 $, which is possible: $\theta = 1/4$, $\delta = 1/8$.

\medskip
Now we dominate the quantity $P_{\epsilon}(x) = |\epsilon^{-1}\ell(1/\epsilon)\phi(\kappa  x/\epsilon)|L_T^x$ for all $|x| > 1$. The second integral is actually supported by the compact set $\{|x|\in[1,S_T]\}$ where $S_T = \sup_{s\in[0,T]}|W_s|$ since $L_T^x = 0$ as soon as $|x| > S_T$. Using \eqref{bbb} with $\delta = 1/2$,
$$P_{\epsilon}(x) \leq A (|x|^{-1}  + 1)(|x|^{1/2} + |x|^{-1/2})L_T^x,$$

\noindent
which is a.s. bounded and thus integrable on $\{|x|\in[1,S_T]\}$.

\medskip\noindent
\textbf{Item (ii): Asymptotics of $\xi_{\epsilon}$}. First remember from Proposition \ref{integrability_f} that $f\in\mathrm{\bm{L}}^{1}(\mu)$ if and only if $\int_1^{\infty}\frac{\dr x}{x\ell(x)} < \infty$.

\medskip
\textit{Case 1:} Let us first suppose $\int_1^{\infty}\frac{\dr x}{x\ell(x)} = \infty$ and show that $\xi_{\epsilon} \underset{\epsilon\to0}{\sim}\kappa (f_+ + f_-)\ell(1/\epsilon)\int_1^{1/\epsilon}\frac{\dr x}{x\ell(x)}$. For $\epsilon \in (0,\kappa)$
\begin{equation}\label{xi_epsilon}
    \begin{split}
        \xi_{\epsilon} & = \kappa \ell(1/\epsilon)\int_{-\kappa /\epsilon}^{\kappa /\epsilon}\phi(x)\dr x\\
         & = \kappa \ell(1/\epsilon)\int_{-\kappa /\epsilon}^{-1}\phi(x)\dr x + \kappa \ell(1/\epsilon)\int_{-1}^{1}\phi(x)\dr x + \kappa \ell(1/\epsilon)\int_{1}^{\kappa /\epsilon}\phi(z)\dr z.
    \end{split}
\end{equation}

\noindent
The middle term is $o(\ell(1/\epsilon)\int_1^{1/\epsilon}\frac{\dr x}{x\ell(x)})$ and will not contribute. Assume e.g. $f_+\neq0$, then $\phi(x)\underset{x\to\infty}{\sim}f_+ / [x\ell(x)]$ and it comes that
$$\kappa \int_{1}^{\kappa /\epsilon}\phi(x)\dr x \underset{\epsilon\to0}{\sim} \kappa  f_+\int_{1}^{\kappa /\epsilon}\frac{\dr x}{x\ell(x)} \underset{\epsilon\to0}{\sim} \kappa  f_+\int_{1}^{1/\epsilon}\frac{\dr x}{x\ell(x)},$$

\noindent
where we used the fact that $L(x) = \int_1^{x}\frac{\dr v}{v\ell(v)}$ is a slowly varying function by Lemma \ref{prop_slow_var}-(iii). If $f_- \neq 0$, then $\kappa \int_{-\kappa /\epsilon}^{-1}\phi(x)\dr x \underset{\epsilon\to0}{\sim} \kappa  f_-\int_{1}^{1/\epsilon}\frac{\dr x}{x\ell(x)}$ as previously. If next  $f_- = 0$, we have $\phi(x)\underset{x\to-\infty}{=}o(|x\ell(|x|)|^{-1})$, then $\kappa \int_{-\kappa /\epsilon}^{-1}\phi(x)\dr x \underset{\epsilon\to0}{=} o(\int_{1}^{1/\epsilon}\frac{\dr x}{x\ell(x)})$ because $\int_1^{\infty}\frac{\dr x}{x\ell(x)} = \infty$.

\medskip
Now, if $f_+ = 0$, then necessarily $f_-\neq0$ and we can proceed as previously.

\medskip
\textit{Case 2:} $\int_1^{\infty}\frac{\dr x}{x\ell(x)} < \infty$. We must show $\xi_{\epsilon} \underset{\epsilon\to0}{\sim}-\kappa (f_+ + f_-)\ell(1/\epsilon)\int_{1/\epsilon}^{\infty}\frac{\dr x}{x\ell(x)}$. In this case, $f\in \mathrm{\bm{L}}^1(\mu)$ and we impose $\mu(f) = 0$ so that $\int_{\mathbb{R}}\phi(x)\dr x = 0$ by \eqref{ccc} and thus
$$\xi_{\epsilon} = -\kappa \ell(1/\epsilon)\int_{-\infty}^{-\kappa /\epsilon}\phi(x)\dr x - \kappa \ell(1/\epsilon)\int_{\kappa /\epsilon}^{\infty}\phi(x)\dr x.$$

\noindent
Let us treat the case $f_+ \neq 0$ and $f_- = 0$, the other cases being treated similarly. Since $f_+\neq 0$, we can integrate the equivalence $\phi(x)\underset{x\to\infty}{\sim}f_+ / [x\ell(x)]$ and get that $\int_{\kappa /\epsilon}^{\infty}\phi(x)\dr x \underset{x\to\infty}{\sim}f_+ \int_{1/\epsilon}^{\infty}\frac{\dr x}{x\ell(x)}$, where we used that $N(x) = \int_{x}^{\infty}\frac{\dr x}{x\ell(x)}$ is slowly varying by Lemma \ref{prop_slow_var}. Next $\int_{-\infty}^{-\kappa /\epsilon}\phi(x)\dr x \underset{x\to-\infty}{=}o( \int_{1/\epsilon}^{\infty}\frac{\dr x}{x\ell(x)})$ since $\phi(x)\underset{x\to-\infty}{=}o(|x\ell(|x|)|^{-1})$.
\end{proof}

We are now ready to give the proof of Theorem \ref{main_theorem}-(iii)-(iv).

\begin{proof}[Proof of Theorem \ref{main_theorem}-(iii)-(iv).]

Consider a Brownian motion $(W_t)_{t\geq0}$, its local time $(L_t^0)_{t\geq0}$ at $0$ and $\tau_t = \inf\{u\geq0, L_u^0 > t\}$. Consider also the process $(K_t)_{t\geq0}$ defined in Subsection \ref{stable_process} for $(a,b) = (f_+, f_-)$.

For each $\epsilon > 0$, consider the processes $(A_t^{\epsilon})_{t\geq}$, $(\tau_t^{\epsilon})_{t\geq}$ and $(H_t^{\epsilon})_{t\geq}$ from Lemma \ref{first_lemma} with the choice $a_{\epsilon} = \epsilon / \kappa $. We know from this Lemma that
\begin{equation}\label{equality_distrib}
    \left(\int_0^{t/\epsilon}f(X_s)\dr s\right)_{t\geq0}\overset{d}{=}\left(H_{\tau_t^{\epsilon}}^{\epsilon}\right)_{t\geq0}
\end{equation}

\noindent
And we know from Lemma \ref{lemma_time_change}-(ii) that $\tau_t^{\epsilon} \longrightarrow \tau_t$ a.s. for each $t$ fixed. Thus $T = \sup_{\epsilon\in(0,1)}\tau_t^{\epsilon}$ is a.s. finite. 

\bigskip\noindent
\textbf{Item (iii)}. We aim at showing that $(\epsilon^{1/\alpha}\ell(1/\epsilon)\int_0^{t/\epsilon}f(X_s)\dr s)_{t\geq0} \overset{f.d.}{\longrightarrow} (\sigma_{\alpha}S_t^{(\alpha)})_{t\geq0}$. By \eqref{equality_distrib} it is enough to show that a.s., 
$$\Delta_t(\epsilon) = \left|\epsilon^{1/\alpha}\ell(1/\epsilon)H_{\tau_t^{\epsilon}}^{\epsilon} - \kappa^{1/\alpha} K_{\tau_t}\right| \underset{\epsilon\to0}{\longrightarrow}0,$$

\noindent
for each fixed $t\geq0$. Indeed, this would imply that a.s., for any $t_1,\ldots,t_n\geq0$, the vector $(\Delta_{t_1}(\epsilon), \ldots,\Delta_{t_n}(\epsilon))$ converges to $0$ which implies the convergence in finite dimensional distribution. Moreover, setting $S_t^{(\alpha)} = \sigma_{\alpha}^{-1}\kappa^{1/\alpha} K_{\tau_t}$, Lemmas \ref{stable0_1} and \ref{stable1_2} tell us that $(S_t^{(\alpha)})_{t\geq0}$ is a stable process such that $$\mathbb{E}[\exp(i\xi S_t^{(\alpha)})] =  \exp\left(-c_{\alpha,a,b}t|\sigma_{\alpha}^{-1}\kappa^{1/\alpha}\xi|^{\alpha}\left(1-i\beta_{\alpha, a, b}\tan\left(\frac{\alpha\pi}{2}\right) \mathrm{sgn}(\xi)\right)\right) =  \exp(-t|\xi|^{\alpha}z_{\alpha}(\xi)),$$

\noindent
where $z_{\alpha}(\xi)) = 1-i\beta_{\alpha, a, b}\tan\left(\frac{\alpha\pi}{2}\right) \mathrm{sgn}(\xi)$ as in the statement.

\medskip
We have
$$\Delta_t(\epsilon) \leq \sup_{s\in[0,T]}\left|\epsilon^{1/\alpha}\ell(1/\epsilon)H_s^{\epsilon} - \kappa^{1/\alpha} K_s\right| + \kappa^{1/\alpha}\left|K_{\tau_t^{\epsilon}} - K_{\tau_t}\right|.$$

\noindent
The first term goes to $0$ from Lemma \ref{lemme2}-(i) and the second by continuity of $(K_t)_{t\geq0}$.

\bigskip\noindent
\textbf{Item (iv)}. We want to show that with $(\xi_{\epsilon})_{\epsilon> 0}$ defined in Lemma \ref{lemme2}-2,
$$\left(\epsilon \ell(1/\epsilon) \int_0^{t/\epsilon}f(X_s)\dr s - \xi_{\epsilon} t\right)_{t\geq0} \overset{f.d.}{\longrightarrow} \left(\sigma_{1}S_t^{(1)}\right)_{t\geq0}.$$

\noindent
Once again, by \eqref{equality_distrib}, it is enough to show that
$$\Delta_t(\epsilon) = \left|\epsilon \ell(1/\epsilon) H_{\tau_t^{\epsilon}}^{\epsilon} - \xi_{\epsilon}t - \kappa K_{\tau_t}\right| \underset{\epsilon\to0}{\longrightarrow}0,$$

\noindent
for each fixed $t\geq 0$. Indeed, setting $S_t^{(1)} = \sigma_{1}^{-1}\kappa K_{\tau_t}$, Lemma \ref{stable_1} tells us that $(S_t^{(1)})_{t\geq0}$ is a stable process such that
\begin{align*}
    \mathbb{E}[\exp(i\xi S_t^{(1)})] & = \exp\left(-c_{a,b}t|\sigma_{1}^{-1}\kappa\xi|\left(1+i\beta_{a, b}\frac{2}{\pi}\log(\sigma_{1}^{-1}\kappa|\xi|) \mathrm{sgn}(\xi)\right) + it\tau_{a,b}\sigma_{1}^{-1}\kappa\xi\right) \\
     & = \exp\left(-t|\xi|\left(1+i\beta_{a, b}\frac{2}{\pi}\log(\sigma_{1}^{-1}\kappa|\xi|) \mathrm{sgn}(\xi) - i \tau_{a,b}\sigma_{1}^{-1}\kappa\,\mathrm{sgn}(\xi)\right)\right) \\
     & = \exp(-t|\xi|z_{1}(\xi)),
\end{align*}

\noindent
where $z_{1}(\xi) = 1+i\beta_{a, b}\frac{2}{\pi}\log(\sigma_{1}^{-1}\kappa|\xi|) \mathrm{sgn}(\xi) - i \tau_{a,b}\sigma_{1}^{-1}\kappa\,\mathrm{sgn}(\xi)$ as in the statement.

\medskip
We write
\begin{equation*}
    \begin{split}
        \Delta_t(\epsilon) & \leq \left|\epsilon \ell(1/\epsilon) H_{\tau_t^{\epsilon}}^{\epsilon} - \xi_{\epsilon}L_{\tau_t^{\epsilon}}^0 - \kappa K_{\tau_t^{\epsilon}}\right|
        + \xi_{\epsilon}\left|A_{\tau_t^{\epsilon}}^{\epsilon} - L_{\tau_t^{\epsilon}}^0\right| + \kappa\left|K_{\tau_t^{\epsilon}} - K_{\tau_t}\right|\\
         & \leq \sup_{t\in[0,T]}\left|\epsilon \ell(1/\epsilon) H_t^{\epsilon} - \xi_{\epsilon}L_t^0 - \kappa K_t\right| + \sup_{t\in[0,T]}\xi_{\epsilon}\left|A_t^{\epsilon} - L_t^0\right|+ \kappa\left|K_{\tau_t^{\epsilon}} - K_{\tau_t}\right|.
    \end{split}
\end{equation*}

\noindent
The first term goes to $0$ from Lemma \ref{lemme2}-(ii) and the last one by continuity of $(K_t)_{t\geq0}$. Now remember that $\xi_{\epsilon} = \kappa \ell(1/\epsilon)\int_{-\kappa /\epsilon}^{\kappa /\epsilon}\phi(x)\dr x = \gamma(1/\epsilon)$ where $\gamma(x) = \kappa \ell(x)\int_{-\kappa x}^{\kappa x}\phi(v)\dr v$. Using the notations of Lemma \ref{prop_slow_var} and the asymptotic estimate of $\xi_{\epsilon}$, see Lemma \ref{lemme2}-(ii), we have
$$\gamma(x)\underset{x\to\infty}{\sim}-\kappa(f_+ + f_-)l(x)N(\kappa x)\quad \text{if $f\in\mathrm{\bm{L}}^1(\mu)$},$$

\noindent
and
$$\gamma(x)\underset{x\to\infty}{\sim}\kappa(f_+ + f_-)l(x)L(\kappa x) \quad \text{if $f\notin\mathrm{\bm{L}}^1(\mu)$}.$$

\noindent
In any case, $\gamma$ is equivalent to the product of of two slowly varying functions, and thus, is a slowly varying function. Hence, Lemma \ref{lemma_time_change}-(iii) tells us that the second term goes to $0$.
\end{proof}

\subsection{The diffusive regime and critical diffusive regime}\label{diffusive_regime}

This case is standard and we use the strategy explained in the introduction. It consists in solving the Poisson equation $\mathcal{L}g = f$, see Jacod-Shiryaev \cite[Chapter VIII section 3f]{jacod2013limit}. We recall the function $g$ defined by
$$g(x) = 2\int_0^x \sca'(v)\int_v^{\infty}f(u)[\sigma^{2}(u)\sca'(u)]^{-1}\dr u\dr v,$$

\noindent
which is a $C^2$ function solving $2bg' + \sigma^2 g'' = -2f$. Also since $(X_t)_{t\geq0}$ is a regular positive recurrent diffusion by Assumption \ref{assump_b_sigma}, we classically deduce that $X_t$ tends in law to $\mu$ as $t$ goes to infinity, see Kallenberg \cite[Theorem 23.15]{kallenberg2006foundations}. We first show Theorem \ref{general_theorem}.

\begin{proof}[Proof of Theorem \ref{general_theorem}.]
Using Itô formula with the function $g$, we have
$$\int_0^t f(X_s)\dr s = \int_0^t g'(X_s)\sigma(X_s)\dr B_s - g(X_t).$$

\noindent
Thus $\epsilon^{1/2}\int_0^{t/\epsilon}f(X_s)\dr s = M_t^{\epsilon} - Y_t^{\epsilon}$ where $M_t^{\epsilon} = \epsilon^{1/2}\int_0^{t/\epsilon}g'(X_s)\sigma(X_s)\dr B_s$ is a local martingale and $Y_t^{\epsilon}  =\epsilon^{1/2}g(X_{t/\epsilon})$. Since $g$ is continuous, $g(X_{t/\epsilon})$ converges in law as $\epsilon\to0$ and thus $Y_t^{\epsilon}$ converges to $0$ in probability, for each fixed $t$. Moreover if $g$ is bounded, we have a.s. $\sup_{t\in[0,T]}Y_t^{\epsilon}\underset{\epsilon\to0}{\longrightarrow}0$.

\medskip
We now show that $(M_t^{\epsilon})_{t\geq0}$ converges in in law as a continuous process to $(\gamma W_t)_{t\geq0}$. It is enough (see Jacod-Shiryaev \cite[Theorem VII-3.11 page 473]{jacod2013limit}) to show that for each $t\geq0$, $\langle M^{\epsilon}\rangle_t \overset{\mathbb{P}}{\to} \gamma^2 t$ as $\epsilon\to0$. But $\langle M^{\epsilon}\rangle_t = \epsilon\int_0^{t/\epsilon}[g'(X_s)\sigma(X_s)]^2 \dr s$ and the result follows from Remark \ref{ergodic_version} since $g'\sigma\in\mathrm{\bm{L}}^2(\mu)$ by assumption.

\medskip
All in all we have proved that $(\epsilon^{1/2}\int_0^{t/\epsilon}f(X_s)\dr s)_{t\geq0}$ converges in finite dimensional distributions to $(\gamma W_t)_{t\geq0}$ and as a continuous process if $g$ is bounded.
\end{proof}

\begin{proof}[Proof of Theorem \ref{main_theorem}-(i)-(ii).] The first point is covered by Theorem \ref{general_theorem} since $g'\sigma\in\mathrm{\bm{L}}^2(\mu)$ by Proposition \ref{g'sigma}. The diffusive constant $\sigma_{\alpha}^2 = \int_{\mathbb{R}}[g'\sigma]^2\dr \mu$ is indeed the one specified in the introduction. Remember that $g'(x) = 2\sca'(x)\int_x^{\infty}f(v)[\sigma^2(v)\sca'(v)]^{-1}\dr v$ and $\mu(\dr x) = \kappa[\sigma^2(x)\sca'(x)]^{-1}\dr x$ so that in the end
$$\sigma_{\alpha}^2 = 4\kappa\int_{\mathbb{R}}\sca'(x)\left(\int_x^{\infty}f(v)[\sigma^{2}(v)\sca'(v)]^{-1}\dr v\right)^2 \dr x.$$

In the case $\alpha = 2$ and $\rho=\infty$, we can use exactly the same proof: we write again that $|\epsilon/\rho_{\epsilon}|^{1/2}\int_0^{t/\epsilon}f(X_s)\dr s = M_t^{\epsilon} - Y_t^{\epsilon}$, where $M_t^{\epsilon} = |\epsilon/\rho_{\epsilon}|^{1/2}\int_0^{t/\epsilon}g'(X_s)\sigma(X_s)\dr B_s$ is a local martingale and $Y_t^{\epsilon}  =|\epsilon/\rho_{\epsilon}|^{1/2}g(X_{t/\epsilon})$. First $Y_t^{\epsilon}$ tends to $0$ in probability for each $t$ fixed, as previously, and since $|\epsilon/\rho_{\epsilon}|\to 0$. Finally it only remains to show that for each $t\geq0$,
$$\langle M^{\epsilon}\rangle_t = |\epsilon/\rho_{\epsilon}|\int_0^{t/\epsilon}[g'(X_s)\sigma(X_s)]^2 \dr s \overset{\mathbb{P}}{\longrightarrow}\sigma_2^2,$$

\noindent
which is the case by Remark \ref{critical_diffusive_remark}.
\end{proof}

\begin{remark}
We stress that it is plausible we can show the limit theorem in the L\'evy regime, at least when $f_+ = -f_-$, using the martingale strategy, i.e. writing $\int_0^{t}f(X_s)\dr s = M_t - g(X_t)$ where $(M_t)_{t\geq0}$ is a local martingale with bracket $\int_0^{t}[g'(X_s)\sigma(X_s)]^2\dr s$. Using similar arguments as in Lemma \ref{lemme2}, we can show that, correctly rescaled, this bracket behaves like a stable subordinator. Hence $M_t$ might behave like a Brownian motion subordinate by a stable subordinator, which is known to be a stable process. 
\end{remark}

\section{Applications}\label{section_applications}

Here we give some examples of application.
Each time, we try to explain intuitively why the resulting limiting stable process is not a Brownian motion.
It can be either because (a) $f$ is large near infinity even if the diffusion \eqref{SDE} has small 
return times to $0$ (as studied by Jara-Komorowski-Olla \cite{ollamilton} in the case of Markov chains), 
or (b) the diffusion \eqref{SDE} has large return times to $0$, meaning that its invariant
distribution has a rather slow decay; (c) or for both reasons.

\medskip
In case (a), and as we will see only in case (a), one easily determines the index $\alpha$ 
of the limiting stable process
from the behavior of $\mu(\{f>x\})$ as $x \to \pm \infty$; namely, $\mu(\{f>x\})\sim |x|^{-\alpha}$,
up to a constant or a slowly varying function. 

\medskip\noindent
\textbf{Diffusions with fast decay invariant measure :} Consider the following stochastic differential equation
\begin{equation}
    X_t = -\frac{\theta + 1}{2}\int_0^t \mathrm{sgn}(X_s)|X_s|^{\theta}\dr s + B_t,
\end{equation}

\noindent
where $\theta > 0$ and $(B_t)_{t\geq0}$ is a Brownian motion. This model includes the Ornstein-Uhlenbeck process ($\theta = 1$). Following equations \eqref{scale_function} and \eqref{speed_measure}, the invariant measure and the scale function are given by
$$\mu(\dr x) = \kappa e^{-|x|^{\theta + 1}}\dr x, \qquad \sca(x) = \int_0^x e^{|v|^{\theta + 1}}\dr v.$$

\noindent
One can check that $\sca(x) \underset{x\to\pm\infty}{\sim}\mathrm{sgn}(x)e^{|x|^{\theta + 1}} / [(\theta  + 1) |x|^{\theta}]$. Hence, as soon as $f$ is a continuous function such that there exists $\alpha > 0$ and $(f_+, f_-)\in\mathbb{R}^2$ satisfying
$$|x|^{(1/\alpha - 2)\theta}e^{- |x|^{\theta + 1} / \alpha}f(x) \underset{x\to\pm\infty}{\longrightarrow}f_{\pm},$$

\noindent
and $|f_+| + |f_-| > 0$, we can apply Theorem \ref{main_theorem} with $\ell \equiv 1$, provided $\int_0^{t/\epsilon}f(X_s)\dr s$ is replaced by $\int_0^{t/\epsilon}[f(X_s) - \mu(f)]\dr s$ when $\alpha > 1$. We can also generalize using slowly varying functions.

\medskip
Roughly, the return times of $(X_t)_{t\geq0}$ have exponential moments, 
because its invariant distribution has an exponential decay. To get an
$\alpha$-stable process at the limit with some $\alpha \in (0,2)$, the function $f$
really has to be large near infinity. Namely, we need that
there exists a slowly varying function $L$ and $c_+, c_- > 0$ such that
\begin{equation}\label{discrete_condition}
    L(x)|x|^{\alpha}\mu(\{f>x\})\underset{x\to\infty}{\longrightarrow}c_+\quad\text{and}\quad L(x)|x|^{\alpha}\mu(\{f<x\})\underset{x\to-\infty}{\longrightarrow}c_+.
\end{equation}

\medskip\noindent
\textbf{A toy kinetic model :} Here we slightly generalize the results of \cite{nasreddine2015diffusion}, \cite{cattiaux2019diffusion}, \cite{lebeau2019diffusion} and \cite{fournier2018one}. Consider a one-dimensional particle with position $X_t\in\mathbb{R}$ and velocity $V_t\in\mathbb{R}$ subject to random shocks and a force field $F(v) = \frac{\beta}{2}\frac{\Theta'(v)}{\Theta(v)}$ for some $\beta > 1$ and a function $\Theta : \mathbb{R}\mapsto(0,+\infty)$ of class $C^1$. The Newton equations describing the motion of the particle are
\begin{equation}\label{kinetic_model_2}
V_t = B_t +\int_0^t F(V_s)\dr s, \qquad X_t = \int_0^t V_s \dr s,
\end{equation}

\noindent
where $(B_t)_{t\geq0}$ is a Brownian motion modeling the random shocks. The invariant measure and the scale function are given by
$$\mu(\dr x) = \kappa[\Theta(x)]^{\beta}\dr x, \qquad \sca(x) = \int_0^x [\Theta(v)]^{-\beta}\dr v.$$

\noindent
We make the following assumption on the function $\Theta$ :

\begin{assumption}\label{assumption_theta}
There exists $(c_-, c_+)\in\mathbb{R}_+^2$ such that $|x|\Theta(x) \longrightarrow c_{\pm}$ as $x\to\pm\infty$ and $c_+ + c_- > 0$.  If $c_+ =0$ (respectively $c_- = 0$), we moreover impose there exists $A > 0$ such that for all $x\geq A$, $\Theta'(x) < 0$ (respectively for all $x\leq -A$, $\Theta'(x) > 0$).
\end{assumption}

\noindent
Obviously, $\lim_{x\to\pm\infty}\sca(x) = \pm\infty$ and since $\beta > 1$, $\mu(\mathbb{R})<\infty$. Now take $f =\text{id}$, we want to find $\alpha > 0$ such that Assumption \ref{assump_f} holds. Let us show that, with $\alpha = (\beta + 1) / 3$, we have 
\begin{align}\label{ttt}
    |\sca'(x)|^{-2}|\sca(x)|^{2-1/\alpha}x \underset{x\to\infty}{\longrightarrow}(\beta  + 1)^{1/\alpha - 2}c_{+}^{\beta/\alpha}, \\
    |\sca'(x)|^{-2}|\sca(x)|^{2-1/\alpha}x \underset{x\to-\infty}{\longrightarrow}-(\beta  + 1)^{1/\alpha - 2}c_{-}^{\beta/\alpha}.\label{yyy}
\end{align}

\medskip\noindent
Observe that we have the following estimations
\begin{equation}\label{kinetic_asympt}
    |x|^{\beta}|\sca'(x)|^{-1}\underset{x\to\pm\infty}{\longrightarrow}c_{\pm}^{\beta}\quad\text{and}\quad|x|^{\beta + 1}|\sca(x)|^{-1}\underset{x\to\pm\infty}{\longrightarrow}(\beta + 1)c_{\pm}^{\beta}.
\end{equation}

\medskip\noindent
Hence if $c_+ > 0$ and $c_- > 0$, the result is immediate. We consider e.g. the case $c_+=0$. First we notice that $\sca''(x) = -\beta\Theta'(x)[\Theta(x)]^{-\beta - 1}$ and thus Assumption \ref{assumption_theta} tells us that there exists $A > 0$, such that $\sca'$ is increasing on $[A,\infty)$. Hence for all $x\geq A$, we have 
$$0\leq\sca(x) \leq \sca(A) + \sca'(x)(x-A)\leq \sca(A) + \sca'(x)x,$$

\noindent
and thus for all $x\geq A$, since $2-1/\alpha > 0$ (because $\beta >1$),
$$|\sca'(x)|^{-2}|\sca(x)|^{2-1/\alpha}x \leq |\sca'(x)|^{-2}|\sca(A) + \sca'(x)x|^{2-1/\alpha}x \underset{x\to\infty}{\sim}|\sca'(x)|^{-1/\alpha}|x|^{3-1/\alpha}$$

\noindent
But since $\alpha = (\beta + 1) / 3$, we have $3-1/\alpha = \beta / \alpha$ and thus $|\sca'(x)|^{-1/\alpha}|x|^{3-1/\alpha} = |x^{\beta}[\sca'(x)]^{-1}|^{1/\alpha}$ which goes to $0$ by \eqref{kinetic_asympt}.

\medskip
We can apply Theorem \ref{main_theorem}: with $\alpha = (\beta + 1) / 3$, $m=\mu(\text{id})$, $\ell\equiv1$, $f_+ = (\beta  + 1)^{1/\alpha - 2}c_{+}^{\beta/\alpha}$ and $f_- = -(\beta  + 1)^{1/\alpha - 2}c_{-}^{\beta/\alpha}$,

\begin{enumerate}[label=(\roman*)]
    \item If $\beta > 5$, $\left(\epsilon^{1/2}(X_{t/\epsilon} - mt/\epsilon)\right)_{t\geq0} \overset{f.d.}{\longrightarrow} \left(\sigma_{\alpha}W_t\right)_{t\geq0}.$
    
    \item If $\beta = 5$, $\left(|\epsilon/\log\epsilon|^{1/2}(X_{t/\epsilon} - mt/\epsilon)\right)_{t\geq0} \overset{f.d.}{\longrightarrow} \left(\sigma_{2}W_t\right)_{t\geq0}.$
    
    \item If $\beta\in(2,5)$, $\left(\epsilon^{1/\alpha}(X_{t/\epsilon} - mt/\epsilon)\right)_{t\geq0} \overset{f.d.}{\longrightarrow} \left(\sigma_{\alpha}S_t^{(\alpha)}\right)_{t\geq0} .$
    
    \item If $\beta = 2$, $\left(\epsilon (X_{t/\epsilon} - \xi_{\epsilon} t/\epsilon)\right)_{t\geq0} \overset{f.d.}{\longrightarrow} \left(\sigma_{\alpha}S_t^{(\alpha)}\right)_{t\geq0} .$
    
    \item If $\beta \in (1,2)$, $\left(\epsilon^{1/\alpha}X_{t/\epsilon}\right)_{t\geq0} \overset{f.d.}{\longrightarrow} \left(\sigma_{\alpha}S_t^{(\alpha)}\right)_{t\geq0} .$
\end{enumerate}

Remark that $\alpha$ ranges from $2/3$ to infinity. The additional assumption that $\Theta$ is eventually monotonic when $c_+ = 0$ (or $c_- = 0$) is not optimal, the true assumption is \eqref{ttt} and \eqref{yyy}.

\medskip
We stress that this model is the starting point of this paper as it was explained in the references section. Considering the \textit{symmetric} function $\Theta(v) = (1+v^2)^{-1/2}$ and reasoning on the law of $(X_t, V_t)$, which is the solution of the associated kinetic Fokker-Planck equation, a series of P.D.E papers (Nasreddine-Puel \cite{nasreddine2015diffusion} : $\beta > 5$, Cattiaux-Nasreddine-Puel \cite{cattiaux2019diffusion} : $\beta = 5$ and Lebeau-Puel \cite{lebeau2019diffusion} : $\beta \in (1, 5)\setminus\{2, 3, 4\}$) answered these questions, finding some symmetric stable processes at the limit. Then, using probabilistic techniques, Fournier and Tardif \cite{fournier2018one} treated all the cases (even when $\beta \in (0,1]$), still in a symmetric context. We thus recover their results and extend them to asymmetrical forces.

\medskip
When $\beta\in(0,1)$, the process $(V_t)_{t\geq0}$ is null recurrent and Fournier and Tardif \cite{fournier2018one} show the rescaled process $(\epsilon^{1/2}V_{t/\epsilon})_{t\geq0}$ converges in finite dimensional distribution to a symmetrized Bessel process and $(\epsilon^{3/2}X_{t/\epsilon})_{t\geq0}$ to an integrated symmetrised Bessel process (their proof actually holds for $\beta\in(-1,1)$). Although we can extend their result to asymmetrical forces, this phenomenon seems specific to this equation and we were not able to extend it to general diffusions.

\medskip
Remark that for this model, the return times do not have exponential moments, because 
the invariant distribution has a slow decay. 
The exponent $\alpha$ is not prescribed by the invariant measure. Indeed we have
$$|x|^{\beta - 1}\mu(\{f>x\})\underset{x\to\infty}{\longrightarrow}\kappa c_+^{\beta}/(\beta-1)\quad \text{and} \quad|x|^{\beta - 1}\mu(\{f<x\})\underset{x\to\infty}{\longrightarrow}\kappa c_-^{\beta}/(\beta-1)$$

\noindent
and $\alpha = (\beta + 1) / 3$, while one might expect $\alpha = \beta -1$. Observe that if $\beta\in(1,2)$, we have $ (\beta + 1) / 3 > \beta - 1$ and if $\beta > 2$, we have  $ (\beta + 1) / 3 < \beta - 1$. The dynamics are a little bit more complicated and one must take into account the fact that the return times can be long.

\bigskip\noindent
\textbf{SDE without drift :} Consider the following SDE
\begin{equation}
    X_t = \int_0^t \left(1 + |X_s|\right)^{\beta / 2}\dr B_s,
\end{equation}

\noindent
with $\beta > 1$. We have $\sca(x) = x$ (indeed $(X_t)_{t\geq0}$ is a martingale) and $\mu(\dr x) = \kappa(1+|x|)^{-\beta}\dr x$. Set for instance $f(x) = x / (1+|x|)^{1-\gamma}$ with $\gamma > \beta - 2$, then $\mu(f) = 0$ and one can see that, with $\alpha = 1 / (\gamma + 2 - \beta) > 0$, $f_+ = 1$ and $f_- = -1$, we have
$$|x|^{2-1/\alpha - \beta}f(x) \underset{x\to\pm\infty}{\longrightarrow}f_{\pm}.$$

\noindent
Hence we can apply Theorem \ref{main_theorem} with $\ell\equiv1$.

\begin{enumerate}[label=(\roman*)]
    \item If $\beta - 2 < \gamma < \beta - 3/2$, then $\left(\epsilon^{1/2}\int_0^{t/\epsilon}f(X_s)\dr s\right)_{t\geq0} \overset{f.d.}{\longrightarrow} \left(\sigma_{\alpha}W_t\right)_{t\geq0}.$
    
    \item If $\gamma = \beta - 3/2$, then $\left(|\epsilon/\log\epsilon|^{1/2}\int_0^{t/\epsilon}f(X_s)\dr s\right)_{t\geq0} \overset{f.d.}{\longrightarrow} \left(\sigma_{2}W_t\right)_{t\geq0}.$
    
    \item If $\beta - 3/2 < \gamma < \beta -1$, then $\left(\epsilon^{1/\alpha}\int_0^{t/\epsilon}f(X_s)\dr s\right)_{t\geq0} \overset{f.d.}{\longrightarrow} \left(\sigma_{\alpha}S_t^{(\alpha)}\right)_{t\geq0} .$
    
    \item If $\gamma = \beta - 1$, then $\left(\epsilon\int_0^{t/\epsilon}f(X_s)\dr s\right)_{t\geq0} \overset{f.d.}{\longrightarrow} \left(\sigma_{\alpha}S_t^{(\alpha)}\right)_{t\geq0} .$
    
    \item If $\gamma > \beta - 1$, then $\left(\epsilon^{1/\alpha}\int_0^{t/\epsilon}f(X_s)\dr s\right)_{t\geq0} \overset{f.d.}{\longrightarrow} \left(\sigma_{\alpha}S_t^{(\alpha)}\right)_{t\geq0} .$
\end{enumerate}

\noindent
Now if $\gamma \leq \beta -2$, one can see that the function $g$, solution of the Poisson equation, is such that
$$g'(x) = \int_x^{\infty}f(v)\sigma^{-2}(v)\dr v \underset{x\to\infty}{\sim}x^{\gamma - \beta + 1}\quad\text{and}\quad g'(x) = \int_{-\infty}^{x}f(v)\sigma^{-2}(v)\dr v \underset{x\to-\infty}{\sim}-|x|^{\gamma - \beta + 1}.$$
\noindent
Hence $g'\in\mathrm{\bm{L}}^2(\dr x)$ since $2(\gamma - \beta + 1) < -1$ and thus $g'\sigma\in\mathrm{\bm{L}}^2(\mu)$ (since $\mu(\dr x) = \sigma^{-2}(x)\dr x$). We can apply Theorem \ref{general_theorem} which tells us that the diffusive regime holds for $\gamma\leq \beta - 2$.

\medskip
Once again, the return times are slow and the index $\alpha$ is not entirely determined
by the invariant measure.
Indeed we have for $\gamma > 0$
$$\mu(\{f>x\})\underset{x\to\infty}{\sim}C|x|^{(1-\beta) / \gamma} \quad \text{and} \quad \mu(\{f<x\})\underset{x\to-\infty}{\sim}C|x|^{(1-\beta) / \gamma}$$
\noindent
for some constant $C$.

\medskip
Assume $\beta = 5/4$ and let $f$ be a function going to $0$ at $\pm\infty$ such that $\mu(f)\neq0$. If $\alpha = 4/3$, we have
$$|x|^{2-1/\alpha - \beta}(f(x)-\mu(f)) \underset{x\to\pm\infty}{\longrightarrow}-\mu(f),$$

\noindent
and thus we can apply Theorem \ref{main_theorem}-(iii) which tells us that
$$\left(\epsilon^{3/4}\int_0^{t/\epsilon}[f(X_s)-\mu(f)]\dr s\right)_{t\geq0} \overset{f.d.}{\longrightarrow} \left(\sigma_{4/3}S_t^{(4/3)}\right)_{t\geq0} .$$

\bibliographystyle{plain}
\bibliography{refs.bib}

\end{document}